\documentclass[reqno]{amsart}
\usepackage{amssymb,amsfonts,amscd,verbatim,hyperref,cleveref,graphics}
\usepackage[latin1]{inputenc}
\usepackage{mathrsfs}
\usepackage{color}
\usepackage{enumitem}
\usepackage[all]{xy}

\numberwithin{equation}{section}

\newtheorem{thm}{Theorem}[section]
\newtheorem{cor}[thm]{Corollary}
\newtheorem{lem}[thm]{Lemma}
\newtheorem{prop}[thm]{Proposition}
\theoremstyle{definition}
\newtheorem{defn}[thm]{Definition}
\newtheorem{rem}[thm]{Remark}

\newtheorem{example}[thm]{Example}
\newtheorem{quest}[thm]{Question}

\renewcommand{\epsilon}{\varepsilon}
\newcommand{\N}{\mathbb{N}}
\newcommand{\Z}{\mathbb{Z}}
\newcommand{\R}{\mathbb{R}}
\newcommand{\C}{\mathbb{C}}

\renewcommand{\le}{\ensuremath{\leqslant}}

\renewcommand{\ge}{\ensuremath{\geqslant}}

\newcommand{\mat}[1]{\begin{bmatrix} #1 \end{bmatrix}}

\newcommand{\ands}{\quad\mbox{and}\quad}

\newcommand{\SC}{{SC}}
\newcommand{\opSC}{\operatorname{SC}}
\newcommand{\scr}{\operatorname{sc}}

\newcommand{\opEAE}{\operatorname{EAE}}
\newcommand{\eae}{\operatorname{eae}}

\newcommand{\ran}{\operatorname{ran}}

\newcommand{\lcm}{\operatorname{lcm}}

\newcommand{\cJ}{{\mathcal J}}

\newcommand{\cW}{{\mathcal W}}\newcommand{\cX}{{\mathcal X}}
\newcommand{\cY}{{\mathcal Y}}\newcommand{\cZ}{{\mathcal Z}}

\newcommand{\GM}{\cX_{\normalfont{\text{GM}}}}

\newcommand{\Go}{\cX_{\normalfont{\text{G}}}}

\newcommand{\sB}{{\mathscr B}}

\newcommand{\sE}{{\mathscr E}}
\newcommand{\sG}{{\mathscr G}}

\newcommand{\sK}{{\mathscr K}}

\newcommand{\sS}{{\mathscr S}}

\newcommand{\sW}{{\mathscr W}}

\begin{document}

\title[EAE and SC for Fredholm operators]{Equivalence after extension and Schur coupling for Fredholm operators on Banach spaces}

\author[S.\ ter Horst]{Sanne ter Horst}
\address{S.\ ter Horst, Department of Mathematics, Research Focus Area:\ Pure and Applied Analytics, North-West University, Potchefstroom, 2531 South Africa and DSI-NRF Centre of Excellence in Mathematical and Statistical Sciences (CoE-MaSS)}
\email{Sanne.TerHorst@nwu.ac.za}

\author[N.J.\ Laustsen]{Niels Jakob Laustsen}
\address{N.J.\ Laustsen, Department of Mathematics and Statistics, Fylde College, Lancaster University, Lancaster, LA1 4YF, United Kingdom}
\email{n.laustsen@lancaster.ac.uk}

\thanks{This work is based on the research supported in part by the National Research Foundation of South Africa (Grant Numbers 90670 and 118513).
}

\begin{abstract}
  Schur coupling (SC) and equivalence after extension (EAE) are important relations for  bounded  operators on Banach spaces. It has been known for 30 years that the former implies the latter, but only recently  Ter Horst, Messer\-schmidt, Ran and Roe\-lands  disproved the converse by   constructing a pair of Fredholm operators  which are EAE, but not SC.

  Motivated by this result, we investigate when EAE and  SC coincide for Fred\-holm operators. Fred\-holm operators which are EAE have the same Fredholm index.
  Surprisingly, we find that for each integer~$k$ and every pair of Banach spaces~$(\cX,\cY)$, either \emph{no} pair of Fredholm operators of index~$k$ acting on~$\cX$ and~$\cY$, respectively, is SC, or \emph{every} pair of this kind which is EAE is also SC.
  Consequently, the question whether EAE and SC coincide for Fredholm operators of index~$k$ depends only on the geometry of the underlying Banach spaces~$\cX$ and~$\cY$, not on the properties of the operators themselves.

We quantify this finding by introducing two numerical indices which capture the coincidence of EAE and SC,  and provide a number of examples illustrating the possible values of these indices. Notably, this includes an example showing that the above-mentioned result of Ter Horst \textit{et al}, which is based on a pair of essentially incomparable Banach spaces, does not extend to projectively incomparable Banach spaces.
\end{abstract}



\maketitle

\section{Introduction}\label{S:Intro}

\noindent
Equivalence after extension (EAE), matricial coupling (MC) and Schur coupling (SC) are three relations for bounded  operators on Banach spaces that originate in the study of Wiener--Hopf integral operators \cite{BGK84} and have found numerous applications since. A~key feature in many of these applications  is that the three relations coincide. This observation led Bart and Tsekanovski\u{\i} \cite{BT94} to ask whether this is always true.  They already knew that EAE and MC are equivalent \cite{BGK84,BT92a} and that SC implies EAE \cite{BT92b,BT94}, so their precise question was whether EAE  implies SC.

Despite numerous results confirming this implication in special cases \cite{BT92b,BT94,BGKR05,tHR13,T14,tHMR15,tHMRRW18},  recently Ter Horst, Messer\-schmidt, Ran and Roe\-lands~\cite{tHMRR19} showed that  EAE does not in general imply SC. Their counterexample uses bounded operators defined on a pair of Banach spaces which is essentially incomparable, in which case EAE (and hence SC) can occur only for Fredholm operators, while SC additionally requires that the  operators have index zero. By contrast, it is known that EAE and SC coincide for Fredholm operators acting on isomorphic Banach spaces \cite[Proposition~6.1(iv)]{tHMR20}.

These results motivated the present work, in which we study EAE and SC for Fredholm operators without imposing any restrictions on the underlying Banach spaces. Further justification for focussing on Fredholm operators comes from the  prominent role this class  plays in many applications of the theory, as the following studies from the last decade evidence:
diffraction problems \cite{CK14,S19}; Wiener--Hopf factorization \cite{S15b, GKR17} and
invertibility of Wiener--Hopf plus Hankel operators \cite{CS13};
truncated Toeplitz operators \cite{CP16,O'L22};
Riemann--Hilbert problems \cite{CM14};
Helmholtz and Sylvester equations \cite{S15a,S17a} and \cite{D21}, respectively;
completeness theorems for integral and differential operators \cite{KLV22};
and problems concerning electrical networks~\cite{BSVWpre}.\smallskip

Before we state our main results, let us introduce some notation and terminology, most of which is standard.
We follow the convention that $\N=\{1,2,3,\ldots\}$ and $\N_0=\{0,1,2,\ldots\}$.
Let $\cX$ and $\cY$ be  Banach spaces, either real or complex, with $\mathbb{K}\in\{\R,\C\}$ denoting  the scalar field. The term ``operator'' means a bounded linear map between Banach spaces, and $\sB(\cX,\cY)$ denotes the Banach space of all operators from $\cX$ to $\cY$. As usual, $\sB(\cX,\cX)$ is abbreviated $\sB(\cX)$. This convention applies whenever we consider sets of operators: Once a subset $\Sigma(\cX,\cY)$ of $\mathscr{B}(\cX,\cY)$ has been defined, we write $\Sigma(\cX)$ instead of $\Sigma(\cX,\cX)$.

The identity operator on a Banach space $\cX$ is denoted by $I_{\cX}$, while  the kernel and the range of an operator $T$ are denoted by $\ker T$ and $\ran T$, respectively. Two Banach spaces $\cX$ and $\cY$ are \emph{isomorphic,} written $\cX\cong\cY$, if~$\sB(\cX,\cY)$ contains a bijection, called an \emph{isomorphism}. The Banach Isomorphism Theorem ensures that the inverse of an isomorphism is automatically bounded.

\begin{defn} Let  $U\in\sB(\cX)$ and $V\in\sB(\cY)$. We say that:
 \begin{enumerate}[label={\normalfont{(\roman*)}}]
  \item $U$ and $V$ are \emph{equivalent after extension}, abbreviated EAE, if there exist Banach spaces $\cX_0$ and $\cY_0$ and isomorphisms $E\in\sB(\cY\oplus \cY_0,\cX\oplus\cX_0)$ and $F\in\sB(\cX\oplus \cX_0,\cY\oplus \cY_0)$ such that
\begin{equation}\label{EAE}
\begin{bmatrix} U&0\\0&I_{ \cX_0} \end{bmatrix} =E\begin{bmatrix} V&0\\0& I_{ \cY_0}\end{bmatrix}F.
\end{equation}
\item $U$ and $V$ are \emph{Schur coupled}, abbreviated SC,
if there exist isomorphisms $A\in\sB(\cX)$ and $D\in\sB(\cY)$ and operators $B\in\sB(\cY,\cX)$ and $C\in\sB(\cX,\cY)$  such that
\begin{equation}\label{SC}
U=A-BD^{-1}C \ands V=D-CA^{-1}B.
\end{equation}
\end{enumerate}
\end{defn}

As noted above, whenever~$U$ and~$V$ are SC, they are also EAE. Motivated by many applications in which the converse implication holds,  Bart and Tsekanovski\u{\i} asked the following question in \cite{BT94}:
\begin{quest}\label{Q:BT}
  Under which conditions on the Banach spaces~$\cX$ and~$\cY$ and/or on the  operators~$U$ and~$V$  is it true that whenever the operators $U\in\sB(\cX)$ and $V\in\sB(\cY)$ are EAE, they are also SC?
\end{quest}

We shall address this question in the case where $U$ and $V$ are Fredholm operators. Before doing so,   let us recall some basic facts about this class of operators. An operator $T\in\sB(\cX,\cY)$  is called a \emph{Fredholm operator} if the quantities
\[
\alpha(T) = \dim\ker T\qquad\text{and}\qquad \beta(T) = \dim\cY/\ran T
\]
are both finite. The latter condition implies that $\ran T$ is closed. As usual, we write~$\Phi(\cX,\cY)$ for the subset of~$\sB(\cX,\cY)$ consisting of Fredholm operators. The \emph{index} of a Fredholm operator~$T$ is defined by
\[ i(T) = \alpha(T)-\beta(T)\in\Z,
\]
and for  $k\in\Z$, $\Phi_k(\cX,\cY)$ denotes the set of $T\in\Phi(\cX,\cY)$ such that $i(T)=k$.

In  the 1990s, Bart and Tsekanovski\u{\i} gave the following  characterization of equivalence after extension for Fredholm operators; see \cite[Theorem~4]{BT92b}, and also \cite[Theorem~6, page 211]{BGKR05}.

\begin{thm}\label{T:EAEfredholm}
Let $U\in\sB(\cX)$ and $V\in\sB(\cY)$ for some Banach spaces $\cX$ and $\cY$.
\begin{enumerate}[label={\normalfont{(\roman*)}}]
\item\label{T:EAEfredholm1} Suppose that $U$ and $V$ are EAE. Then $U$ is a Fredholm operator if and only if $V$  is a Fredholm operator.
\item\label{T:EAEfredholm2} Suppose that $U$ and $V$ are Fredholm operators.
Then $U$ and $V$ are EAE  if and only if
\[
\alpha(U)=\alpha(V)\qquad \text{and}\qquad \beta(U) = \beta(V).
\]
In particular, Fredholm operators which are EAE have the same index.
\end{enumerate}
\end{thm}

As a consequence, the following sets, defined  for every $k\in\Z$ and every pair of Banach spaces $(\cX,\cY)$, provide the natural setting in which to study \Cref{Q:BT} for Fredholm operators:
\begin{equation}\label{EAEkSCk}
\begin{aligned}
  \opEAE_k(\cX,\cY) &= \{ (U,V)\in\Phi_k(\cX)\times\Phi_k(\cY) : U\ \text{and}\ V\ \text{are EAE}\},\\
  \opSC_k(\cX,\cY) &= \{ (U,V)\in\Phi_k(\cX)\times\Phi_k(\cY) : U\ \text{and}\ V\ \text{are SC}\}.
\end{aligned}
\end{equation}
In view of \Cref{T:EAEfredholm}\ref{T:EAEfredholm2},  the former set can alternatively be written as
 \begin{equation}\label{eq:EAEalt} \opEAE_k(\cX,\cY)=\{(U,V)\in\Phi_k(\cX)\times\Phi_k(\cY) : \alpha(U)=\alpha(V)
   \}. \end{equation}
 These sets are useful in our investigation because they allow us to express the statement that EAE and SC are equivalent for every pair of Fredholm operators of index~$k$ on~$\cX$ and~$\cY$, respectively, in the concise form
 $\opSC_k(\cX,\cY) = \opEAE_k(\cX,\cY)$, where we note that  the inclusion $\opSC_k(\cX,\cY)\subseteq \opEAE_k(\cX,\cY)$ is always true because SC implies EAE.

 Using this notation, we can state easily two important results that motivated our work. First,
 the answer to \Cref{Q:BT} is always affirmative for Fredholm operators of index~$0$ (see \cite[Theorem~3]{BT92b} and \cite[Theorem~5]{BGKR05}).   In the above notation, this simply means that
\begin{equation}\label{BGKR:eq}
   \opSC_0(\cX,\cY)=\opEAE_0(\cX,\cY)
\end{equation}
for every pair of Banach spaces~$(\cX,\cY)$.

Second, we can state the seminal result of Ter Horst, Messer\-schmidt, Ran and Roe\-lands~\cite{tHMRR19} showing
that there are pairs of Fredholm operators which are EAE, but not SC.  This requires the following piece of terminology.

\begin{defn}\label{D:EssIncomp} A pair of  Banach spaces $(\cX,\cY)$ is \emph{essentially incomparable} if \mbox{$I_\cX-ST\in\Phi(\cX)$} for every $S\in\sB(\cY,\cX)$ and $T\in\sB(\cX,\cY)$.
\end{defn}

\begin{thm}\label{T:tHMRR}
  \begin{enumerate}[label={\normalfont{(\roman*)}}]
\item\label{T:tHMRR1} Let $(\cX,\cY)$ be a pair of essentially incomparable Banach spaces.
  Then $U\in\sB(\cX)$ and $V\in\sB(\cY)$ are SC if and only if $(U,V)\in\opEAE_0(\cX,\cY)$.
\item\label{T:tHMRR2} There exists a pair of essentially incomparable Banach spaces $(\cX,\cY)$ such that
  \[ \opEAE_k(\cX,\cY)\ne\emptyset\quad \text{for every}\quad k\in\Z. \]
    Hence  EAE and SC are not equivalent for Fredholm operators of non-zero index on such Banach spaces.\\
  An example is given by $\cX=\ell_p$ and $\cY=\ell_q$ for $1\le p<q<\infty$.
\end{enumerate}
\end{thm}

 The significance of analyzing whether SC and EAE are equivalent for each value~$k$ of the Fredholm index separately will become clear from the next result, which is the first main outcome of our work. To state it concisely, we introduce a numerical index $\eae(\cX,\cY)$ as follows:
 Set $\mathbb{I}_\Phi(\cX) = \{k\in\Z : \Phi_k(\cX)\neq\emptyset\}$ and then define
\begin{equation}\label{eq:eaeDefn} \eae(\cX,\cY) = \begin{cases} 0\ &\text{if}\ \mathbb{I}_\Phi(\cX)\cap\mathbb{I}_\Phi(\cY)\cap\N = \emptyset,\\
    \min\mathbb{I}_\Phi(\cX)\cap\mathbb{I}_\Phi(\cY)\cap\N\ &\text{otherwise.} \end{cases} \end{equation}

\begin{thm}\label{T:SCdichotomy}
Let $k\in\Z$, and let~$\cX$ and~$\cY$ be Banach spaces.
\begin{enumerate}[label={\normalfont{(\roman*)}}]
\item\label{T:SCdichotomy0} $\opEAE_k(\cX,\cY)\ne\emptyset$ if and only if $k$ is a multiple of~$\eae(\cX,\cY)$.
\item\label{T:SCdichotomy2}  Suppose that~$k$ is a multiple of~$\eae(\cX,\cY)$. Then $\opSC_k(\cX,\cY) = {\opEAE}_k(\cX,\cY)$ if and only if $\opSC_k(\cX,\cY)\ne\emptyset$.
\item\label{T:SCdichotomy1} Suppose that~$k$ is not a multiple of~$\eae(\cX,\cY)$. Then $\Phi_k(\cX)=\emptyset$ or \mbox{$\Phi_k(\cY)=\emptyset$}, and consequently $\opSC_k(\cX,\cY)=\opEAE_k(\cX,\cY)=\emptyset$.
\end{enumerate}
\end{thm}

The most remarkable part of \Cref{T:SCdichotomy} is the implication $\Leftarrow$ in~\ref{T:SCdichotomy2} which, when written out, states that as soon as \emph{one} pair of operators \mbox{$(U,V)\in\Phi_k(\cX)\times\Phi_k(\cY)$} is SC, then EAE and SC are equivalent for \emph{all} pairs $(U,V)\in\Phi_k(\cX)\times\Phi_k(\cY)$. In other words,  equivalence of EAE and SC for Fredholm operators depends only on the geometry of the underlying Banach spaces~$\cX$ and~$\cY$ and on the Fredholm index~$k$, not on the Fredholm operators themselves.

In view of \Cref{T:SCdichotomy}\ref{T:SCdichotomy2}, it would be of great interest to establish a counterpart  of \Cref{T:SCdichotomy}\ref{T:SCdichotomy0}  for SC.
In analogy with~\eqref{eq:eaeDefn}, we introduce the set
\begin{equation}\label{eq:Isc} \mathbb{I}_{\text{SC}}(\cX,\cY) = \{ k\in\Z : \opSC_k(\cX,\cY)\ne\emptyset\} \end{equation}
and the associated index
\begin{equation}\label{eq:sc} \operatorname{sc}(\cX,\cY) = \begin{cases} 0\ &\text{if}\ \mathbb{I}_{\text{SC}}(\cX,\cY)\cap\N=\emptyset,\\
    \min \mathbb{I}_{\text{SC}}(\cX,\cY)\cap\N\ &\text{otherwise.} \end{cases} \end{equation}
Combining \Cref{T:SCdichotomy}\ref{T:SCdichotomy0} with  the inclusion $\opSC_k(\cX,\cY)\subseteq \opEAE_k(\cX,\cY)$, we see that $\mathbb{I}_{\text{SC}}(\cX,\cY)\subseteq\eae(\cX,\cY)\Z$. By \Cref{T:SCdichotomy}\ref{T:SCdichotomy2}--\ref{T:SCdichotomy1},  our main question --- whether EAE and SC are equivalent for every pair of Fredholm operators on~$\cX$ and~$\cY$, respectively --- boils down to whether
$\mathbb{I}_{\text{SC}}(\cX,\cY)=\eae(\cX,\cY)\Z$.
We  address this question in the following proposition.

\begin{prop}\label{P:SC1806}
  Let~$\cX$ and~$\cY$ be Banach spaces. Then $\opSC_k(\cX,\cY)=\opEAE_k(\cX,\cY)$ for every $k\in\Z$ if and only if $\operatorname{sc}(\cX,\cY)=\eae(\cX,\cY)$.

In general, $\operatorname{sc}(\cX,\cY) = n\eae(\cX,\cY)$  for some $n\in\N_0$, and the following chain of inclusions holds:
\begin{equation}\label{SC1806incl}
\begin{aligned}
\operatorname{sc}(\cX,\cY)\Z & \subseteq  \mathbb{I}_{\normalfont{\text{SC}}}(\cX,\cY)
=  \{ k\in\Z : \opSC_k(\cX,\cY)
= \opEAE_k(\cX,\cY)\ne\emptyset\}\\
& \subseteq\eae(\cX,\cY)\Z
=  \{ k\in\Z : \opEAE_k(\cX,\cY)\ne\emptyset\}
=  \mathbb{I}_{\Phi}(\cX)\cap\mathbb{I}_{\Phi}(\cY).
\end{aligned}
\end{equation}
\end{prop}

\begin{rem}
The first part of Proposition~\ref{P:SC1806} implies that the second inclusion in~\eqref{SC1806incl} is an equality if and only if $\operatorname{sc}(\cX,\cY)=\eae(\cX,\cY)$, in which case the first inclusion  is also an equality. In fact, we do not know any instances where the first inclusion in~\eqref{SC1806incl} is proper, and we conjecture that it may always be an equality; see \Cref{S:SC} for a more detailed discussion of this question.
\end{rem}

  We conclude this overview of our main findings with some results that illustrate the values which the numerical indices~$\eae(\cX,\cY)$ and~$\operatorname{sc}(\cX,\cY)$  can take when various incomparability conditions are imposed on the Banach spaces~$\cX$ and~$\cY$. This work is motivated by, and closely related to,
  the seminal result of Ter Horst, Messer\-schmidt, Ran and Roe\-lands that we stated in  \Cref{T:tHMRR}.
  We begin with a result whose first part is simply a restatement of \Cref{T:tHMRR}\ref{T:tHMRR1},  while  its second part contains \Cref{T:tHMRR}\ref{T:tHMRR2} as a special case, corresponding to $k_0=1$.

\begin{thm}\label{T:essinNew}
  \begin{enumerate}[label={\normalfont{(\roman*)}}]
\item\label{T:essinNew1} Let~$(\cX,\cY)$ be a pair of essentially incomparable Banach spaces. Then $\operatorname{sc}(\cX,\cY)=0$.
\item\label{T:essinNew2} For every $k_0\in\N_0$, there exists a pair of essentially incomparable Banach spaces $(\cX,\cY)$ such that $\eae(\cX,\cY) = k_0$.
  \end{enumerate}
\end{thm}

\Cref{T:essinNew}\ref{T:essinNew1} immediately raises the question whether we can weaken the hypo\-thesis  that the pair $(\cX,\cY)$ is essentially incomparable   without losing the conclusion that  $\operatorname{sc}(\cX,\cY)=0$.  The most obvious, very modest weakening would be to assume that~$\cX$ and~$\cY$ are projec\-tive\-ly in\-comparable in the following sense.

\begin{defn}\label{D:ProjIncomp} A pair of  Banach spaces $(\cX,\cY)$ is \emph{projectively incomparable} if no infinite-dimensional, complemented subspace of~$\cX$ is isomorphic to a complemented subspace of $\cY$.
\end{defn}

However, it turns out that  this hypothesis is too weak to imply that $\operatorname{sc}(\cX,\cY)=0$, as our next result will show. It also contains some information about the possible values of the indices~$\eae(\cX,\cY)$ and~$\operatorname{sc}(\cX,\cY)$.

\begin{thm}\label{thm:improjEAE=SC}
\begin{enumerate}[label={\normalfont{(\roman*)}}]
\item\label{T:improjEAE=SCv1} For every  $k_0\in\N$, there exists a pair of projectively incomparable Banach spaces~$(\cX,\cY)$ such that $\operatorname{sc}(\cX,\cY)=\eae(\cX,\cY)=k_0$.
\item\label{T:improjEAE=SCv2} For every  $k_0\in\N$, there exists a pair of projectively incomparable Banach spaces~$(\cX,\cY)$ such that    $\eae(\cX,\cY)=1$ and 
$\operatorname{sc}(\cX,\cY)=k_0$.
\end{enumerate}
\end{thm}

\Cref{thm:improjEAE=SC} is highly surprising because the difference between essential and projective incomparability is very subtle, as evidenced by the fact that it took nearly 30 years to find an example which distinguishes them. Indeed, Tarafdar \cite{T72a,T72b} asked  in 1972 whether projective incomparability implies  essential  incomparability, having noted that the converse is true, but it was not until 2000 that  Aiena and Gon\-z\'{a}\-lez  answered this question by giving a counterexample (see \cite[Proposition~3.7]{AG}). It  relied on a sophisticated Banach space constructed by Gowers and Maurey \cite{gm2}. To the best of our knowledge, no simpler examples have subsequently been found.
We shall discuss  the  relationship between essential and projective incomparability in more detail in \Cref{S:Incomp}.

In view of \Cref{thm:improjEAE=SC}, let us consider what may happen when~$\cX$ and~$\cY$ are \emph{not}  projectively incomparable. Then they contain isomorphic, infinite-dimensional complemented subspaces; that is, $\cX$ and~$\cY$ admit decompositions of the form
\begin{equation}\label{notPI}
\cX = \cX_1\oplus \cX_2\qquad\text{and}\qquad \cY=\cY_1\oplus \cY_2\quad\text{with}\quad \cX_2\cong \cY_2,
\end{equation}
where $\cX_2$ and $\cY_2$ are infinite-dimensional. The next proposition answers the question when is $\eae(\cX,\cY)=\scr(\cX,\cY)$ in the case where $\cX_1$, $\cX_2$ and~$\cY_1$ are pair\-wise essentially incomparable? Its statement involves the greatest common divisor (gcd) of two quantities that could potentially both be~$0$, in which case the gcd is not defined. We fix this issue by adopting the convention that $\gcd(0,0)=0$.

\begin{prop}\label{P:beyondprojincomp}
Let~$\cX$  and~$\cY$ be Banach spaces that decompose as in~\eqref{notPI}, and suppose that each of the pairs $(\cX_1,\cX_2)$ and $(\cY_1,\cY_2)$ is essentially incomparable. Then
\begin{equation}\label{nonPIeaeform}
\eae(\cX,\cY)=\gcd(\eae(\cX_1,\cY_1),\eae(\cX_2,\cY_2)).
\end{equation}
Suppose additionally that the pair $(\cX_1,\cY_1)$ is  essentially incomparable. Then
\begin{equation}\label{nonPIscform}
\scr(\cX,\cY) =\scr(\cX_2,\cY_2) = \eae(\cX_2,\cY_2).
\end{equation}
In particular, EAE and SC coincide for all pairs of Fredholm operators on $\cX$ and~$\cY$, i.e., $\eae(\cX,\cY)=\scr(\cX,\cY)$, if and only if $\eae(\cX_2,\cY_2)$ divides $\eae(\cX_1,\cY_1)$.
\end{prop}

Note that  we do not demand that the isomorphic subspaces~$\cX_2$ and~$\cY_2$ are infinite-dimensional in \Cref{P:beyondprojincomp}. Therefore part~\ref{T:essinNew1} of \Cref{T:essinNew} appears as a special case of \Cref{P:beyondprojincomp} corresponding to $\cX_2=\cY_2=\{0\}$, while the case where $\cX$ and $\cY$ are isomorphic is obtained by taking $\cX_1=\cY_1=\{0\}$.

In analogy with~\Cref{thm:improjEAE=SC}, it turns out that the second part of \Cref{P:beyondprojincomp}  is no longer true if we  replace the  hypothesis that  the pair $(\cX_1,\cY_1)$ is essentially incomparable with the weaker hypothesis that it is projectively incomparable.

\begin{thm}\label{T:improjUnited}
  \begin{enumerate}[label={\normalfont{(\roman*)}}]
 \item\label{T:improjV1} For every  $k_0\in\N$, there exist infinite-dimensional Banach spaces~$\cX_1$, $\cY_1$ and~$\cZ$ such that:
  \begin{enumerate}[label={\normalfont{(\arabic*)}}]
  \item\label{T:improjV1i} The pair $(\cX_1,\cY_1)$ is projectively incomparable.
  \item\label{T:improjV1ii} The pairs $(\cX_1,\cZ)$ and $(\cY_1,\cZ)$ are essentially incomparable.
  \item\label{T:improjV1iv} $\eae(\cZ,\cZ) = 0$.
  \item\label{T:improjV1iii} The Banach spaces $\cX = \cX_1\oplus \cZ$ and $\cY=\cY_1\oplus\cZ$ satisfy
    \[ \operatorname{sc}(\cX,\cY) = \eae(\cX,\cY)=k_0. \]
\end{enumerate}
\item\label{T:improjV2}
  For every  $k_0\in\N$, there exist infinite-dimensional Banach spaces~$\cX_1$, $\cY_1$ and~$\cZ$ satisfying~\ref{T:improjV1i}--\ref{T:improjV1iv} above, and such that the Banach spaces $\cX = \cX_1\oplus \cZ$ and $\cY=\cY_1\oplus\cZ$ satisfy
  $\eae(\cX,\cY)=1$ and 
  $\operatorname{sc}(\cX,\cY)=k_0$.
\end{enumerate}
\end{thm}

\begin{rem}\label{R:EssIncVsProjInc}
  Let us compare and contrast \Cref{T:improjUnited} with  \Cref{P:beyondprojincomp}. To align notation, note that $\cX_2=\cY_2=\cZ$. \Cref{T:improjUnited}\ref{T:improjV1ii} implies that~\eqref{nonPIeaeform} holds true. However, \eqref{nonPIscform}  fails for  the pair $(\cX,\cY)$ in  both parts~\ref{T:improjV1} and~\ref{T:improjV2} of \Cref{T:improjUnited} because they satisfy $\operatorname{sc}(\cX,\cY)=k_0\ne 0=\eae(\cZ,\cZ)$. This difference is due to the fact that~\eqref{nonPIscform} requires that the pair $(\cX_1,\cY_1)$ is essentially incomparable, but we only know that it is  projectively incomparable in \Cref{T:improjUnited}.
\end{rem}

\subsection*{Conclusion}
Prior to this paper, at the level of general Banach spaces, the only conclusive results regarding the question whether EAE and SC coincide for Fredholm operators were that they do if the underlying spaces are isomorphic, and that there exist examples where they do not if the  underlying spaces are essentially incomparable. We have shown that the result for  essentially incomparable spaces does not carry over to projectively incomparable spaces (see Theorem \ref{thm:improjEAE=SC}), despite the fact that the difference between these two incomparability notions is very subtle.

In the case where the Banach spaces~$\cX$ and $\cY$ admit decompositions of the form~\eqref{notPI} in which the sub\-spaces~$\cX_1$, $\cY_1$ and~$\cX_2\,(\cong\cY_2)$ are pairwise essentially incomparable, the question whether EAE and SC coincide for Fredholm operators is completely resolved in \Cref{P:beyondprojincomp}; the answer can be expressed in terms of the values of the Fredholm index of operators on the subspaces~$\cX_1$, $\cY_1$ and~$\cX_2$. \Cref{T:improjUnited} shows that, once again, this result does not carry over to projectively  incomparable spaces; see \Cref{R:EssIncVsProjInc} for details.

The question that remains is whether one can always find a decomposition of the form~\eqref{notPI} in which the subspaces~$\cX_1$, $\cY_1$ and~$\cX_2$ are pairwise essentially incomparable. Unfortunately, this is not possible, even if we replace ``essentially incomparable'' with ``projectively incomparable'',  as we shall see in \Cref{cor:nomaxprojincomp} and \Cref{P:Ex2}.

The above results rely on the remarkable observation in \Cref{T:SCdichotomy} that, for Banach spaces $\cX$ and $\cY$ and $k\in\Z$, either no pair of operators \mbox{$(U,V)\in\Phi_k(\cX)\times \Phi_{k}(\cY)$} is SC, or a pair of this kind which is SC exists, in which case the set of all such pairs that are SC is the same as the set of all such pairs that are EAE. This means that the question whether EAE and SC coincide for Fredholm operators on~$\cX$ and~$\cY$ reduces to determining the sets of indices for which pairs of Fredholm operators on~$\cX$ and~$\cY$ with these particular indices that are EAE or SC, respectively, exist.

Our analysis of these sets led us to define the numerical indices $\eae(\cX,\cY)$ and $\scr(\cX,\cY)$ which satisfy that $\eae(\cX,\cY)=\scr(\cX,\cY)$ if and only if EAE and SC coincide for all Fredholm operators on $\cX$ and $\cY$. We have computed their values in various cases; see Theorems~\ref{T:essinNew}, \ref{thm:improjEAE=SC} and~\ref{T:improjUnited}.

\subsection*{Organization}
We conclude this introduction with a brief outline of how the remainder of this paper is organized. It  consists of six sections, including the present.

In \Cref{S:Incomp} we elaborate on the incomparability notions for Banach spaces introduced in Definitions~\ref{D:EssIncomp} and~\ref{D:ProjIncomp}, focussing on their connections with certain classes of operators.  \Cref{S:IntroProofs} contains a characterization of when the set $\opEAE_k(\cX,\cY)$ is non-empty for Banach spaces~$\cX$ and~$\cY$ and $k\in\Z$, and also the proof of Equation~\eqref{nonPIeaeform}.
It turns out to be much more complicated to obtain a similar characterization for the non-emptiness of the set $\opSC_k(\cX,\cY)$, and only a partial analogue is obtained in \Cref{S:SC}, where we also prove \Cref{P:SC1806} and the remainder of~\Cref{P:beyondprojincomp}. These results rely strongly on a novel characterization of the existence of Schur-coupled Fredholm operators of a given index that we establish in \Cref{S:TechThm}. This characterization may be viewed as the fundamental new insight of the paper. \Cref{T:SCdichotomy} is also proved in \Cref{S:TechThm}.

Finally, in \Cref{sec:improj} we use some of the ``exotic'' Banach spaces constructed by Gowers and Maurey, together with ideas from subsequent work of Aiena, Gonz\'{a}lez and Ferenczi, to prove Theorems~\ref{T:essinNew}, \ref{thm:improjEAE=SC} and~\ref{T:improjUnited}.

\section{Incomparability notions for Banach spaces and their connection to operator theory}\label{S:Incomp}

\noindent
The notions of essential and projective incomparability of a pair of Banach spaces will play a key role in the  final section of this paper, where Theorems~\ref{T:essinNew}, \ref{thm:improjEAE=SC} and~\ref{T:improjUnited} are proved. However,  there are certain related notions and results that we shall require be\-fore\-hand. For that reason, we survey the relevant  material at this point.

The  formal definitions of   essential and projective incomparability were already given in Definitions~\ref{D:EssIncomp} and~\ref{D:ProjIncomp}, respectively. We refer to \cite[Section~7.5]{A04} for a much more com\-pre\-hen\-sive treatment of them than we can give here. Indeed, we shall consider only the aspects that we require later, namely their relationship and certain connections to operator theory. This will involve a third incomparability notion, which is stronger, older and arguably more ``natural'' than the other two. It is defined as follows.

\begin{defn}\label{D:totalincomparability}
A pair of Banach spaces~$(\cX,\cY)$ is \emph{totally incomparable} if no closed, infinite-dimensional subspace of~$\cX$ embeds isomorphically into~$\cY$.
\end{defn}

 Totally incomparable Banach spaces are clearly  projectively incomparable, and  it is well known and not hard to see that the converse fails; for instance, $\ell_2$~embeds into~$L_1[0,1]$, but not complementably, so~$L_1[0,1]$ and~$\ell_2$ are projectively incomparable without being totally incomparable.

However, more is true, namely that essential incomparability lies between these two properties, in the sense that total incomparability implies  essential incomparability, which in turn implies  projective incomparability. The easiest way to explain this goes via the following two operator-theoretic notions, which will also be useful elsewhere in this work.

\begin{defn}\label{D:ssIness}  Let~$\cX$ and~$\cY$ be Banach spaces. An operator $T\in\sB(\cX,\cY)$ is:
  \begin{enumerate}[label={\normalfont{(\roman*)}}]
  \item\label{D:ss} \emph{strictly singular} if, for every $\epsilon>0$, every infinite-dimensional subspace~$\cW$ of~$\cX$ contains a unit vector~$w$ such that $\lVert Tw\rVert\le\epsilon$. In other words, the restriction of~$T$ to~$\cW$ is not an isomorphism onto its range.
  \item\label{D:Iness} \emph{inessential} if $I_\cX-ST\in\Phi(\cX)$ for every operator $S\in\sB(\cY,\cX)$.
 \end{enumerate}
\end{defn}

We write $\sS(\cX,\cY)$ and $\sE(\cX,\cY)$ for the collections of strictly singular and inessential operators, respectively, from~$\cX$ to~$\cY$. They generalize the ideal $\sK(\cX,\cY)$ of compact operators in several ways. The following remark lists the main properties that we require.
\begin{rem}\label{R:ss_iness}
\begin{enumerate}[label={\normalfont{(\roman*)}}]
\item\label{R:ss_iness1} Every compact operator is strict\-ly singular, and every strict\-ly singular operator is inessential (see for instance \cite[Theorems~7.36 and 7.44]{A04} or \cite[\S{}26.7.3]{pie}).
\item The assignments~$\sS$ and~$\sE$ are closed operator ideals in the sense of Pietsch (see for instance \cite[page~388 and Theorem~7.5]{A04} or \cite[Theorem~1.9.4, Pro\-p\-o\-si\-tion~4.2.7 and Section~4.3]{pie}).
\item\label{inessperturb} For every $k\in\Z$, the class of Fredholm operators of index~$k$ is stable under inessential perturbations in the sense that $U+T\in\Phi_k(\cX,\cY)$   whenever  $U\in\Phi_k(\cX,\cY)$ and $T\in\sE(\cX,\cY)$ (see for instance \cite[Theorem~7.23]{A04}).
\end{enumerate}
\end{rem}

Comparing Definitions~\ref{D:EssIncomp} and~\ref{D:ssIness}\ref{D:Iness}, we see that a pair of Banach spaces~$(\cX,\cY)$ is essentially incomparable if and only if every operator from~$\cX$ to~$\cY$ is inessential. Both definitions display an obvious lack of symmetry, which raises the question what happens if we replace the condition  that  $I_\cX-ST\in\Phi(\cX)$ with \mbox{$I_\cY-TS\in\Phi(\cY)$} in either of them?
It turns out that it makes no difference because the following well-known elementary lemma  implies that $I_\cX-ST\in\Phi(\cX)$ if and only if \mbox{$I_\cY-TS\in\Phi(\cY)$.}

\begin{lem}\label{L:IndexEqual}
Let $S\in\sB(\cY,\cX)$ and $T\in\sB(\cX,\cY)$ for some Banach spaces~$\cX$ and~$\cY$. Then $\ker(I_\cX - ST)\cong \ker(I_\cY - TS)$. Moreover, if $\ran(I_\cX - ST)$ and $\ran(I_\cY - TS)$ are closed, then $\cX/\ran (I_\cX - ST)\cong \cY/\ran (I_\cY - TS)$.
\end{lem}

\begin{proof}
The earliest mention of the first part of this result that we know of is \cite[Satz~1]{pietsch1963}, while both parts can be found in \cite[S\"{a}tze~2.4-A--2.5-A]{mertins}. Alternatively, the result follows from \cite[page~211, properties~1 and~6]{BGKR05}, see also \cite[Proposition 1]{BT92a}, because the operators $I_\cX - ST$ and $I_\cY - TS$ are Schur coupled via $A=I_\cX$, $B=S$, $C=T$ and $D= I_\cY$.
\end{proof}

Using \Cref{D:ssIness} and \Cref{R:ss_iness}\ref{R:ss_iness1}, we can explain the relationship between total, essential and projective incomparability of a pair of Banach spaces~$(\cX,\cY)$ as follows.
\begin{rem}\label{R:incompnotions}
\begin{enumerate}[label={\normalfont{(\roman*)}}]
\item\label{R:incompnotions1}  If~$\cX$ and~$\cY$ are totally incomparable, then clearly every operator between~$\cX$ and~$\cY$ is strictly singular and therefore inessential, so~$\cX$ and~$\cY$ are essentially incomparable.
\item If~$\cX$ and~$\cY$ are essentially incomparable, then they are also projectively incomparable. Indeed, suppose contrapositively that~$\cX$ and~$\cY$ contain isomorphic, complemented infinite-dimensional subspaces. Then it is easy to construct operators  $S\in\sB(\cY,\cX)$ and  $T\in\sB(\cX,\cY)$  such that $I_\cX-ST$ is a projection with infinite-dimensional kernel and therefore not a Fredholm operator. Hence~$\cX$ and~$\cY$ are not essentially incomparable. We refer to \cite[Theorem~7.69]{A04} and the paragraph following  Definition 7.102 for further details.
\end{enumerate}
\end{rem}

The fact that~$\sS$ and~$\sE$ are operator ideals has the following important consequence, which we shall use repeatedly without further reference.
Suppose that we express an operator $T\colon \cX_1\oplus\cX_2\to\cY_1\oplus\cY_2$ between two direct sums of Banach spaces as an operator-valued matrix in the usual way, that is,
\[ T = \begin{bmatrix} T_{11} & T_{12}\\ T_{21} & T_{22}    \end{bmatrix},\quad\text{where}\quad T_{ij}\in\sB(\cX_j,\cY_i)\quad \text{for}\quad i,j\in\{1,2\}. \]
Then $T$ is strictly singular (respectively, inessential) if and only if $T_{11}$, $T_{12}$, $T_{21}$ and $T_{22}$ are strictly singular (respectively, inessential).\smallskip

We conclude this section by answering a natural question about
a pair of Banach spaces~$(\cX,\cY)$ which is \emph{not} projectively incomparable. This material will not play any role in the remainder of the paper;  we have included it simply because the question is very natural in our context.
Negating the definition of projective incomparability, we see that~$\cX$ and~$\cY$ admit decompositions of the form~\eqref{notPI},
where the isomorphic subspaces~$\cX_2$ and~$\cY_2$ are infinite-dimensional.
If the subspaces~$\cX_1$ and~$\cY_1$ fail to be  projectively incomparable, then they
contain isomorphic, complemented, infinite-dimensional subspaces.  One may wonder whether \emph{all} such subspaces  can somehow be ``transferred'' to~$\cX_2$ and~$\cY_2$, respectively, leading to the following question:

\begin{quest}
  Let~$\cX$ and~$\cY$ be Banach spaces which are not projectively incomparable. Is it always possible to find decompositions of the form~\eqref{notPI}, where 
  the subspaces~$\cX_1$ and~$\cY_1$ are projectively incomparable?
\end{quest}

The answer to this question is ``no''. We shall present two short examples showing this.
In the first, we consider Banach spaces which are $c_0$-direct sums of certain sequences of finite-dimensional Banach spaces. The formal definition is as follows. The \emph{$c_0$\nobreakdash-direct sum} of a sequence $(\cX_n)_{n\in\N}$ of Banach spaces is
\[ \Bigl(\bigoplus_{n=1}^\infty\cX_n\Bigr)_{c_0} = \bigl\{ (x_n)_{n\in\N} : x_n\in\cX_n\ (n\in\N)\ \text{and}\ \lVert x_n\rVert\to 0\ \text{as}\ n\to\infty\bigr\}, \]
endowed with the pointwise vector-space operations and with the norm given by $\lVert (x_n)\rVert = \sup_{ n\in\N}\lVert x_n\rVert$.

Bourgain, Casazza, Lindenstrauss and Tzafriri \cite[\S{}8]{bclt} classified the complemented subspaces of   $\bigl(\bigoplus_{n=1}^\infty\cX_n\bigr)_{c_0}$ in certain cases, including the following result.
\begin{thm}\label{BCLT} Let $\cZ = \bigl(\bigoplus_{n=1}^\infty\ell_1^n\bigr)_{c_0}$ or $\cZ=\bigl(\bigoplus_{n=1}^\infty\ell_2^n\bigr)_{c_0}$, and let~$\cW$ be a complemented, infinite-dimensional subspace of~$\cZ$. Then either \mbox{$\cW\cong c_0$} or $\cW\cong \cZ$.
\end{thm}

\begin{cor}\label{cor:nomaxprojincomp}
   Suppose that~$\cX=\bigl(\bigoplus_{n=1}^\infty\ell_1^n\bigr)_{c_0}$ and~$\cY=\bigl(\bigoplus_{n=1}^\infty\ell_2^n\bigr)_{c_0}$ are decomposed as in \eqref{notPI}, with $\cX_2\cong\cY_2$ infinite-dimensional. Then $\cX_1\cong \cX$ and $\cY_1\cong \cY$.  In particular~$\cX_1$ and~$\cY_1$  both contain a complemented subspace isomorphic to~$c_0$, so they are not projectively incomparable.
\end{cor}
\begin{proof} This follows immediately from \Cref{BCLT} because~$\cX$ and~$\cY$ are not isomorphic to each other or to~$c_0$, so we must have $\cX_2\cong\cY_2\cong c_0$.
\end{proof}

Our second example is similar, but uses only reflexive Banach spaces.
It relies on the following well-known, important properties of $L_p[0,1]$ for $1<p<\infty$.
\begin{thm}\label{T:LpBackground} Let $p\in(1,\infty)$.
\begin{enumerate}[label={\normalfont{(\roman*)}}]
\item\label{T:LpBackground1} The Banach space $L_p[0,1]$ is primary; that is, if it is decomposed into a direct sum of two closed subspaces, then (at least) one of them is isomorphic to~$L_p[0,1]$.
\item\label{T:LpBackground2}  Let $q\in(1,\infty)$. Then $L_p[0,1]$ contains a complemented subspace which is isomorphic to $L_q[0,1]$ if and only if
 $q=p$ or $q=2$.
\end{enumerate}
\end{thm}

\begin{proof} \ref{T:LpBackground1}.  This is shown in~\cite{AEO}.

  \ref{T:LpBackground2}. This follows easily from \cite[Theorem~6.4.21]{ak}.
\end{proof}

\begin{prop}\label{P:Ex2}
Let $\cX = L_p[0,1]$ and $\cY=L_q[0,1]$ for distinct $p,q\in(1,2)\cup(2,\infty)$, and suppose that $\cX$ and $\cY$ are decomposed as in~\eqref{notPI},  with $\cX_2\cong\cY_2$ infinite-dimensional. Then $\cX_1\cong \cX$ and $\cY_1\cong \cY$.  In particular~$\cX_1$ and~$\cY_1$  both contain a complemented subspace isomorphic to~$L_2[0,1]$, so they are not projectively incomparable.
\end{prop}
\begin{proof} Since $L_p[0,1]$ is primary, either $\cX_1\cong\cX$ or $\cX_2\cong \cX$. However, the latter is impossible by \Cref{T:LpBackground}\ref{T:LpBackground2}
because $\cX_2\cong\cY_2$, which is a complemented subspace of~$L_q[0,1]$, where $q\notin\{2,p\}$. Therefore  $\cX_1\cong\cX$. The proof that $\cY_1\cong \cY$ is similar.

Another application of \Cref{T:LpBackground}\ref{T:LpBackground2} shows that~$L_p[0,1]$ and~$L_q[0,1]$ both contain a complemented subspace which is isomorphic to~$L_2[0,1]$. \end{proof}

\section{Non-emptiness of the set $\opEAE_k(\cX,\cY)$}\label{S:IntroProofs}
\noindent
The main purpose of this section is to prove the following characterization of the integers~$k$ for which the set~$\opEAE_k(\cX,\cY)$ defined in~\eqref{EAEkSCk} is non-empty. This is a natural starting point for our investigation because when $\opEAE_k(\cX,\cY)$ is empty, it is obviously equal to its subset $\opSC_k(\cX,\cY)$.

\begin{prop}\label{P:EAEfredholmIndex} The following three conditions are equivalent for every pair of Banach spaces~$(\cX,\cY)$ and every $k\in\Z\colon$
\begin{enumerate}[label={\normalfont{(\roman*)}}]
\item\label{P:EAEfredholmIndex2} $\Phi_k(\cX)\ne\emptyset$ and  $\Phi_k(\cY)\ne\emptyset$.
\item\label{P:EAEfredholmIndex1} $\opEAE_k(\cX,\cY)\neq \emptyset$.
\item\label{P:EAEfredholmIndex3} $k\in \eae(\cX,\cY)\Z$.
\end{enumerate}
\end{prop}

The proof of \Cref{P:EAEfredholmIndex} is relatively simple, using only~\eqref{eq:EAEalt} and two lemmas that follow from basic Fredholm theory. Both of these lemmas are almost certainly known to specialists; we include their (short) proofs for completeness.

\begin{lem}\label{perturbkernel1} Let  $T\in\Phi(\cX,\cY)$ for some infinite-di\-men\-sional Banach spaces~$\cX$ and~$\cY$. Then, for every $m\in\N_0\cap[i(T),\infty)$, there exists a finite-rank operator \mbox{$R\in\sB(\cX,\cY)$} such that $\alpha(T+R)=m$.
\end{lem}

\begin{proof} We consider three cases:\\
{\it Case 1:} If $\alpha(T)=m$, then we can simply take $R=0$.\\
{\it Case 2:} If $\alpha(T)>m$, set $n=\alpha(T)-m\in\N$, and note that \mbox{$\beta(T) = \alpha(T)-i(T)\ge n$}, so we can find an $n$-dimensional subspace~$\cZ$ of~$\cY$ such that $\cZ\cap\ran T=\{0\}$. Take an operator $A\in\sB(\ker T,\cY)$ with range~$\cZ$ and a bounded linear projection~$P$ of~$\cX$ onto~$\ker T$, and define $R=AP\in\sB(\cX,\cY)$. Then $\ker (T+R)=\ker A$, which has dimension $\alpha(T)-\dim\ran A=m$.\\
{\it Case 3:} If $\alpha(T)<m$, choose an $m$-dimensional subspace~$\cW$ of~$\cX$ such that \mbox{$\ker T\subseteq\cW$}, and let $P\in\sB(\cX)$ be a projection onto~$\cW$. Then $\ker T(I_\cX-P) = \cW$, so $R=-TP$ has the required property.
\end{proof}

\begin{rem}
  The condition that $m\ge i(T)$  is necessary in \Cref{perturbkernel1}  because \[ \alpha(T+R)\ge i(T+R)=i(T) \] for every finite-rank operator~$R\in\sB(\cX,\cY)$. \end{rem}

\begin{lem}\label{L:indexideal}
For every Banach space~$\cX$, the set
\[ \mathbb{I}_\Phi(\cX) = \{ k\in\Z : \Phi_k(\cX)\ne\emptyset\} \]
is an ideal of~$\Z$.
\end{lem}

\begin{proof}
The result is clear if~$\cX$ is finite-dimensional because $\mathbb{I}_\Phi(\cX)=\{0\}$ in this case. In general the set~$\mathbb{I}_\Phi(\cX)$ contains~$0$ because $I_\cX\in\Phi_0(\cX)$.  Moreover, it
  is closed under addition and under multiplication by positive integers because
  the Index Theorem implies that $ST\in\Phi_{k+m}(\cX)$ and $T^n\in\Phi_{mn}(\cX)$ for $S\in\Phi_k(\cX)$ and \mbox{$T\in\Phi_m(\cX)$} whenever $k,m\in\Z$  and   \mbox{$n\in\N$}.

  It remains to show that $-k\in\mathbb{I}_\Phi(\cX)$ whenever $k\in\mathbb{I}_\Phi(\cX)$. Suppose that \mbox{$\Phi_k(\cX)\ne\emptyset$} for some $k\in\Z$. If $k\ge 0$, \Cref{perturbkernel1} implies that $\Phi_k(\cX)$ contains a surjection $T$. Since $\ker T$ is finite-dimensional, it follows that~$T$ has a right inverse, which must be a Fredholm operator of index~$-k$. A similar argument works for $k\le 0$, except that we find that $\Phi_k(\cX)$ contains an injection which has a left inverse.
\end{proof}

\begin{cor}\label{C:IphiXcapIphiY} Let~$\cX$ and~$\cY$ be Banach spaces. Then $\mathbb{I}_\Phi(\cX)\cap\mathbb{I}_\Phi(\cY) = \eae(\cX,\cY)\Z$.
\end{cor}

\begin{proof}
\Cref{L:indexideal} implies that $\mathbb{I}_\Phi(\cX)\cap\mathbb{I}_\Phi(\cY)$ is an ideal of~$\Z$, which is a principal ideal domain, so $\mathbb{I}_\Phi(\cX)\cap \mathbb{I}_\Phi(\cY)=n\Z$ for some $n\in\N_0$. It is clear from the definition~\eqref{eq:eaeDefn}
of $\eae(\cX,\cY)$ that $n=\eae(\cX,\cY)$.
\end{proof}

\begin{rem}\label{R:eae_gcd}  Elaborating on these ideas, we obtain an alternative formula for $\eae(\cX,\cY)$, which will be useful later. Indeed,  for every Banach space~$\cX$, the ideal~$\mathbb{I}_\Phi(\cX)$ has a unique non-negative generator, which we shall denote by~$\gamma(\cX)$; in other words, $\gamma(\cX)\in\N_0$ is the unique number such that $\mathbb{I}_\Phi(\cX) = \gamma(\cX)\Z$.

  It follows from elementary number theory that, for every pair of Banach spaces~$\cX$ and~$\cY$, the ideal $\mathbb{I}_\Phi(\cX)\cap \mathbb{I}_\Phi(\cY) = \gamma(\cX)\Z\cap\gamma(\cY)\Z$ is generated by the lowest common multiple (lcm) of~$\gamma(\cX)$ and~$\gamma(\cY)$, provided that we set $\lcm(n,0)=\lcm(0,n)=0$ for every $n\in\Z$.  Combining this result with \Cref{C:IphiXcapIphiY}, we see that
  \begin{equation}\label{eq:eaeDefnAlt} \eae(\cX,\cY) = \lcm(\gamma(\cX),\gamma(\cY)).
  \end{equation}
\end{rem}

\begin{proof}[Proof of Proposition~{\normalfont{\ref{P:EAEfredholmIndex}}}]
 Conditions~\ref{P:EAEfredholmIndex2} and~\ref{P:EAEfredholmIndex3} are equivalent because
  \[ \bigl(\Phi_k(\cX)\neq\emptyset\ \text{and}\ \Phi_k(\cY)\neq\emptyset\bigr)\ \Longleftrightarrow\ k\in\mathbb{I}_\Phi(\cX)\cap \mathbb{I}_\Phi(\cY)\ \Longleftrightarrow\ k\in\eae(\cX,\cY)\Z,  \]
where the final bi-implication follows from \Cref{C:IphiXcapIphiY}.

We complete the proof by showing that conditions~\ref{P:EAEfredholmIndex2} and~\ref{P:EAEfredholmIndex1} are also equivalent.
  Here, the implication  \ref{P:EAEfredholmIndex1}$\Rightarrow$\ref{P:EAEfredholmIndex2} is obvious  because $\opEAE_k(\cX,\cY)\subseteq\Phi_k(\cX)\times \Phi_k(\cY)$ by definition.
  Conversely,  suppose  that $\Phi_k(\cX)$ and $\Phi_k(\cY)$ are both non-empty. If~$\cX$ or~$\cY$ is finite-dimensional,  then necessarily $k=0$, in which case $\opEAE_k(\cX,\cY)$ is non-empty because it contains $(I_\cX,I_\cY)$. Otherwise we may apply \Cref{perturbkernel1} to find operators $U\in\Phi_k(\cX)$ and $V\in\Phi_k(\cY)$ such that $\alpha(U)=\alpha(V)$. This implies that
  $\opEAE_k(\cX,\cY)$ is non-empty because it contains $(U,V)$ by~\eqref{eq:EAEalt}.
\end{proof}

\begin{rem}
  The index $\eae(\cX,\cY)\in\N_0$
  satisfies: \mbox{$\eae(\cX,\cY)=0$} if and only if $\Phi_k(\cX) = \emptyset$ for every $k\in\N$ or $\Phi_k(\cY) = \emptyset$ for every $k\in\N$.

Indeed, the implication $\Leftarrow$ is immediate from the definition of~$\eae(\cX,\cY)$.   Conversely, suppose that $\Phi_k(\cX)\ne\emptyset$ and $\Phi_m(\cY)\ne\emptyset$ for some \mbox{$k,m\in\N$}.  Then the Index Theorem implies that $\Phi_n(\cX)\ne\emptyset$ and $\Phi_n(\cY)\ne\emptyset$ for every  common multiple~$n\in\N$ of~$k$ and~$m$, so $\eae(\cX,\cY)\ge 1$. (Alternatively, this follows from~\eqref{eq:eaeDefnAlt} because $\lcm(\gamma(\cX),\gamma(\cY))=0$ if and only if $\gamma(\cX)=0$ or $\gamma(\cY)=0$.)
\end{rem}

\Cref{P:EAEfredholmIndex} naturally raises the question: What are the possible values of  the index $\eae(\cX,\cY)\in\N_0$? Obviously, $\eae(\cX,\cY) =0$ if~$\cX$ or~$\cY$ is finite-dimensional. Theorems~\ref{T:essinNew}\ref{T:essinNew2}, \ref{thm:improjEAE=SC}\ref{T:improjEAE=SCv1} and~\ref{T:improjUnited}\ref{T:improjV1}
show that all non-negative integers can be realized as $\eae(\cX,\cY)$  for infinite-dimensional Banach spaces~$\cX$ and~$\cY$ satisfying  various additional conditions. At this stage, let us use the index~$\gamma(\cX)$ from \Cref{R:eae_gcd} to verify that  all non-negative integers  can be realized as $\eae(\cX,\cX)$. Although this example may appear simpler than the three above-mentioned theorems,  ultimately they all rely on the same family of ``exotic'' Banach spaces constructed by Gowers and Maurey  in~\cite{gm2}.

 \begin{example}
   We claim that,  for every  $k_0\in\N_0$, there exists an infinite-dimensional Banach space~$\cX_{k_0}$ such that $\gamma(\cX_{k_0})=k_0$.  Consequently
  $\eae(\cX_{k_0},\cX_{k_0})=k_0$ by~\eqref{eq:eaeDefnAlt}, and  \Cref{P:EAEfredholmIndex} shows that
    ${\opEAE}_k(\cX_{k_0},\cX_{k_0})\neq \emptyset$  if and only if $k\in k_0\Z$.

  To verify this claim for $k_0=0$, we require an infinite-dimensional Banach space on which all Fredholm operators have index~$0$.  Gowers and Maurey  constructed such a Banach space in~\cite{gm1}. We shall encounter another space with this property in \Cref{T:GowersHyperplane}.

For $k_0=1$, any Banach space which is isomorphic to its hyperplanes satisfies the claim. Virtually every infinite-dimensional Banach space known prior to 1990 has this property.

Finally, to see that  the claim is true for every $k_0\ge 2$, we use a family of Banach spaces which Gowers and Maurey  constructed in~\cite[\S{}(4.3)]{gm2}.
These Banach spaces will play a key role in \Cref{sec:improj}, where \Cref{T:GMspace} summarizes their main properties; we refer to~\ref{T:GMspace4} for the particular result required at this point.
\end{example}

We conclude this section with two easy observations. The first is that the index~$\gamma(\cX)$, and therefore the associated quantity~$\eae(\cX,\cY)$, is an isomorphic invariant in the following precise sense.

\begin{lem}\label{L:isoInvarianceIphi}
  Let $\cX_1$ and $\cX_2$ be isomorphic Banach spaces. Then $\gamma(\cX_1) = \gamma(\cX_2)$.

  Consequently, if~$(\cY_1,\cY_2)$ is another pair of isomorphic Banach spaces, then \[ \eae(\cX_1,\cY_1)=\eae(\cX_2,\cY_2). \]
\end{lem}

\begin{proof} Let $R\in\sB(\cX_1,\cX_2)$ be an isomorphism. For each $k\in\mathbb{I}_\Phi(\cX_1)$, we can take $T\in\Phi_k(\cX_1)$. The Index Theorem implies that $RTR^{-1}\in\Phi_k(\cX_2)$, so \mbox{$k\in\mathbb{I}_\Phi(\cX_2)$}. This proves that $\mathbb{I}_\Phi(\cX_1)\subseteq\mathbb{I}_\Phi(\cX_2)$. The opposite inclusion follows by inter\-chang\-ing~$\cX_1$ and~$\cX_2$. Hence the ideals~$\mathbb{I}_\Phi(\cX_1)$ and~$\mathbb{I}_\Phi(\cX_2)$ are equal, so they must have the same non-negative generator; that is,  $\gamma(\cX_1) = \gamma(\cX_2)$.

  The final clause is an immediate consequence of~\eqref{eq:eaeDefnAlt}.
\end{proof}

Our second easy observation will be the key ingredient in the proof of the first part of \Cref{P:beyondprojincomp}, as we shall show immediately after it.

\begin{lem}\label{L:essincFredholmIndex} Let $\cX=\cX_1\oplus\cX_2$ be a Banach space.
\begin{enumerate}[label={\normalfont{(\roman*)}}]
\item\label{L:essincFredholmIndex1}  $\gamma(\cX)$  divides $\gcd(\gamma(\cX_1),\gamma(\cX_2))$.
\item\label{L:essincFredholmIndex2}    Suppose that~$\cX_1$ and~$\cX_2$  are essentially incomparable.  Then
  \begin{equation}\label{L:essincFredholmIndex:eq2} \gamma(\cX) = \gcd(\gamma(\cX_1),\gamma(\cX_2)).
  \end{equation}
\end{enumerate}
\end{lem}

\begin{proof} \ref{L:essincFredholmIndex1}. Suppose that $k_j\in\mathbb{I}_\Phi(\cX_j)$ for $j\in\{1,2\}$,  and take $T_j\in\Phi_{k_j}(\cX_j)$. Then we have
\[  \mat{T_1 & 0\\ 0 & T_2}\in\Phi_{k_1+k_2}(\cX), \]
so $k_1+k_2\in\mathbb{I}_\Phi(\cX)$. This shows that
\begin{equation}\label{L:essincFredholmIndex:eq}
\mathbb{I}_\Phi(\cX_1)+\mathbb{I}_\Phi(\cX_2)\subseteq\mathbb{I}_\Phi(\cX). \end{equation}
Recall that we have defined $\gcd(0,0)=0$. This ensures that the formula $m\Z+n\Z = \gcd(m,n)\Z$ holds true for all values of $m,n\in\Z$. Using it, we can rewrite the identity~\eqref{L:essincFredholmIndex:eq}  as $\gcd(\gamma(\cX_1),\gamma(\cX_2))\Z\subseteq \gamma(\cX)\Z$,  which proves~\ref{L:essincFredholmIndex1}.

\ref{L:essincFredholmIndex2}.
Suppose that the subspaces~$\cX_1$ and~$\cX_2$  are essentially incomparable. Then, for each \mbox{$k\in\mathbb{I}_\Phi(\cX)$}, we can take
  \[ T = \begin{bmatrix} T_{11} & T_{12}\\ T_{21} & T_{22} \end{bmatrix}\in\Phi_k(\cX), \]
  where $T_{12}$ and~$T_{21}$ are inessential. In view of \Cref{R:ss_iness}\ref{inessperturb}, this implies that
  \[ \Phi_k(\cX)\ni T - \begin{bmatrix} 0 & T_{12}\\ T_{21} & 0 \end{bmatrix} = \begin{bmatrix} T_{11} & 0\\ 0 & T_{22} \end{bmatrix}, \]
  which in turn means that $T_{11}\in\Phi(\cX_1)$ and $T_{22}\in\Phi(\cX_2)$ with \[ k=i(T_{11})+i(T_{22})\in\mathbb{I}_\Phi(\cX_1)+\mathbb{I}_\Phi(\cX_2). \]
 Hence    $\mathbb{I}_\Phi(\cX_1)+\mathbb{I}_\Phi(\cX_2) =\mathbb{I}_\Phi(\cX)$, and~\eqref{L:essincFredholmIndex:eq2} follows.
\end{proof}

\begin{proof}[Proof of  Equation~\eqref{nonPIeaeform}]\label{proofofEq:nonPIeaeform}
Since the pairs $(\cX_1,\cX_2)$ and $(\cY_1,\cY_2)$ are essentially incomparable,  \Cref{L:essincFredholmIndex}\ref{L:essincFredholmIndex2} implies that
  $\gamma(\cX) = \gcd(\gamma(\cX_1),\gamma(\cX_2))$ and  $\gamma(\cY) = \gcd(\gamma(\cY_1),\gamma(\cY_2))$,  where we observe that $\gamma(\cX_2)=\gamma(\cY_2) = \eae(\cX_2,\cY_2)$ because $\cX_2\cong \cY_2$.
Therefore, applying~\eqref{eq:eaeDefnAlt} twice and using the standard identity \[ \lcm(\gcd(a,c),\gcd(b,c))=\gcd(\lcm(a,b),c)\qquad (a,b,c\in\Z), \]
we obtain
\begin{align*} \eae(\cX,\cY) &= \lcm(\gcd(\gamma(\cX_1),\gamma(\cX_2)),  \gcd(\gamma(\cY_1),\gamma(\cX_2)))\\ &= \gcd(\lcm(\gamma(\cX_1),\gamma(\cY_1)),\gamma(\cX_2)) = \gcd(\eae(\cX_1,\cY_1),\eae(\cX_2,\cY_2)), \end{align*}
as required.
\end{proof}

\section{The key technical theorem and the deduction of \Cref{T:SCdichotomy} from it}\label{S:TechThm}
\noindent
The main aim of this section is to prove the following theorem, which will be the key tool in the remainder of our investigation, and is arguably the most important new insight in the paper, despite the somewhat technical nature of con\-di\-tions~\ref{T:SCfredholmChar1}--\ref{T:SCfredholmChar4}. Notably,  \Cref{T:SCdichotomy} is an easy consequence of it (using also \Cref{P:EAEfredholmIndex}), as we shall show at the end of this section.

\begin{thm}\label{T:SCfredholmChar}
  The following four conditions are equivalent for every  pair of Banach spaces $(\cX,\cY)$ and every $k\in\Z\colon$
\begin{enumerate}[label={\normalfont{(\roman*)}}]
\item\label{T:SCfredholmChar1}
There exist operators $S\in\sB(\cY,\cX)$ and $T\in\sB(\cX,\cY)$ such that $I_\cX - ST\in\Phi_k(\cX)$.
\item\label{T:SCfredholmChar4}
There exist operators $S\in\sB(\cY,\cX)$ and  $T\in\sB(\cX,\cY)$ such that $I_\cY - TS\in\Phi_k(\cY)$.
\item\label{T:SCfredholmChar2}
  $\opEAE_k(\cX,\cY) = \opSC_k(\cX,\cY)$, and this  set is non-empty.
\item\label{T:SCfredholmChar3}   $\opSC_k(\cX,\cY)\ne\emptyset$.
\end{enumerate}
\end{thm}

The proof of \Cref{T:SCfredholmChar} involves two lemmas. The first of these reformulates SC in a way that is much closer to condition~\ref{T:SCfredholmChar1} above.

\begin{lem}\label{char:SC} Let $U\in\sB(\cX)$ and $V\in\sB(\cY)$ for some Banach spaces~$\cX$ and~$\cY$. Then~$U$ and~$V$  are SC if and only if there are   isomorphisms $M\in\sB(\cX)$ and $N\in\sB(\cY)$ and operators $S\in\sB(\cY,\cX)$ and $T\in\sB(\cX,\cY)$  such that
 \begin{equation}\label{char:SC:eq}  UM = I_\cX - ST\qquad\text{and}\qquad VN =  I_\cY - TS. \end{equation}
\end{lem}

\begin{proof} This is a straightforward verification. On the one hand, if the operators $A$, $B$, $C$ and~$D$ satisfy~\eqref{SC}, then $M= A^{-1}$, $N = D^{-1}$, $S= BD^{-1}$ and $T=CA^{-1}$ satisfy~\eqref{char:SC:eq}, and on the other, if $M$, $N$, $S$ and~$T$ satisfy~\eqref{char:SC:eq}, then $A=M^{-1}$, $B=SN^{-1}$, $C=TM^{-1}$ and $D=N^{-1}$ satisfy~\eqref{SC}.
\end{proof}

The  second lemma can be viewed as a technical refinement of \Cref{perturbkernel1}. Its  proof is an adaption of the proof of \cite[Lemma~5.10]{tHMR20}.

\begin{lem}\label{perturbkernel2}
  Let $\cX$ and $\cY$ be Banach spaces, and  suppose that  $I_\cX-S_1T_1\in\Phi_k(\cX)$ for  some $k\in\Z\setminus\{0\}$ and some operators $S_1\in\sB(\cY,\cX)$ and $T_1\in\sB(\cX,\cY)$. Then, for every
  $m\in\N_0\cap[k,\infty)$, there are operators $S_2\in\sB(\cY,\cX)$
    and $T_2\in\sB(\cX,\cY)$ such that $S_1T_1-S_2T_2$ is a
    finite-rank operator and $I_\cX -S_2T_2\in\Phi_k(\cX)$ with $\alpha(I_\cX -S_2T_2)=m$.
\end{lem}

\begin{proof}
We begin by observing that~$\cX$ must be infinite-dimensional because it admits a Fredholm operator of non-zero index. We can therefore apply \Cref{perturbkernel1} to find a finite-rank operator $R\in\sB(\cX)$ such that
\begin{equation}\label{PertAlpha}
\alpha(I_\cX -S_1T_1-R)=m.
\end{equation}
Take a finite-dimensional subspace~$\cW$ of~$\cX$ such that  $\ran R\subseteq\cW$ and $\dim\cW=n\lvert k\rvert$ for some $n\in\N$, and let $R_0\in\sB(\cX,\cW)$  denote the operator~$R$ regarded as a map into~$\cW$. Moreover, let  $J\in\sB(\cW,\cX)$ be the inclusion map.

\Cref{L:IndexEqual} shows that $I_\cY-T_1S_1\in\Phi_k(\cY)$. This implies that $\Phi_{n|k|}(\cY)$ is non-empty by \Cref{L:indexideal}, and therefore $\cY\cong \cY\oplus\cW$ by \cite[Proposition~4.2]{tHMR20}. Take an isomorphism  $L\in\sB(\cY,\cY\oplus\cW)$. Then the operators
\[ S_2 = \begin{bmatrix} S_1 & J\end{bmatrix}L\in\sB(\cY,\cX)\qquad\text{and}\qquad T_2 = L^{-1}\begin{bmatrix} T_1\\ R_0\end{bmatrix}\in\sB(\cX,\cY) \]
 satisfy $S_2T_2 = S_1T_1 +R$. It follows that $S_1T_1-S_2T_2=-R$ is a finite-rank operator, and $I_\cX -S_2T_2\in\Phi_k(\cX)$ because finite-rank perturbations do not change the Fredholm index. Finally, \eqref{PertAlpha} shows that $\alpha(I_\cX -S_2T_2)=m$.
\end{proof}

\begin{proof}[Proof of Theorem {\normalfont{\ref{T:SCfredholmChar}}}]
\Cref{L:IndexEqual} shows that conditions~\ref{T:SCfredholmChar1} and~\ref{T:SCfredholmChar4} are equivalent.

\ref{T:SCfredholmChar1}$\Rightarrow$\ref{T:SCfredholmChar2}. For $k=0$, \eqref{BGKR:eq} shows that  $\opEAE_0(\cX,\cY) = \SC_0(\cX,\cY)$, and this set is non-empty because  it contains $(I_\cX,I_\cY)$.

 Hence it suffices to consider the case $k\ne 0$.
  Suppose that  $I_\cX - S_1T_1\in\Phi_k(\cX)$ for some operators $S_1\in\sB(\cY,\cX)$ and $T_1\in\sB(\cX,\cY)$. Then obviously $\Phi_k(\cX)$ is non-empty, and $\Phi_k(\cY)$ is also non-empty by \Cref{L:IndexEqual}, so $\opEAE_k(\cX,\cY)$  is non-empty by \Cref{P:EAEfredholmIndex}.

  Suppose that  $(U,V)\in\opEAE_k(\cX,\cY)$, so that $\alpha(U)=\alpha(V)$ by \eqref{eq:EAEalt}. Call this number~$m$, and note that $m\ge k$. Our strategy is to modify the operators~$S_1$ and~$T_1$ to obtain a pair for which we can construct isomorphisms~$M\in\sB(\cX)$ and~$N\in\sB(\cY)$ such that~\eqref{char:SC:eq} is satisfied.

  We begin by  applying \Cref{perturbkernel2} to find operators  $S_2\in\sB(\cY,\cX)$ and $T_2\in\sB(\cX,\cY)$ such that $S_1T_1-S_2T_2$ has finite rank and $I_\cX - S_2T_2\in\Phi_k(\cX)$ with $\alpha(I_\cX - S_2T_2)=m$. Then
$\beta(I_\cX - S_2T_2)=m-k=\beta(U)$, so $\ran (I_\cX-S_2T_2)$ and  $\ran U$ are closed subspaces of the same finite codimension in~$\cX$, and therefore we can
take  an isomorphism
$A\in\sB(\cX)$ such that
\begin{equation}\label{T:SCfredholmChar:eq2}
A[\ran (I_\cX-S_2T_2)] =  \ran U.
\end{equation}
\Cref{L:IndexEqual} implies that $\beta(I_\cY - T_2S_2)=\beta(I_\cX - S_2T_2)=m-k=\beta(V)$, so we can also find an isomorphism $B\in\sB(\cY)$ such that
\begin{equation}\label{T:SCfredholmChar:eq1}
B[\ran(I_\cY-T_2S_2)]=\ran V.
\end{equation}

Set $S_3 = AS_2B^{-1}\in\sB(\cY,\cX)$ and $T_3 = BT_2A^{-1}\in\sB(\cX,\cY)$, and observe that these operators satisfy
\begin{equation}\label{T:SCfredholmChar:eq3}  I_\cX-S_3T_3 = A(I_\cX-S_2T_2)A^{-1}. \end{equation}
This implies that $\alpha(I_\cX-S_3T_3)=\alpha(I_\cX-S_2T_2)=m=\alpha(U)$, which is finite, so we can take an isomorphism $M_1\in\sB(\ker(I_\cX-S_3T_3),\ker U)$.
Choose closed subspaces~$\cX_1$ and~$\cX_2$ of $\cX$ such that $\cX =  \ker(I_\cX-S_3T_3)\oplus \cX_1$ and $\cX= \ker U\oplus \cX_2$, and let
\[ R\colon x\mapsto (I_\cX-S_3T_3)x,\ \cX_1\to \ran (I_\cX-S_3T_3),\quad\text{and}\quad U_0\colon x\mapsto Ux,\ \cX_2\to \ran  U, \]
be the restrictions of $I_\cX-S_3T_3$ and~$U$, respectively.  The choices of $\cX_1$ and $\cX_2$ imply that $R$ and $U_0$ are isomorphisms.

Using~\eqref{T:SCfredholmChar:eq3} and~\eqref{T:SCfredholmChar:eq2}, we see that $\ran (I_\cX-S_3T_3) =  \ran  U$, so we can define an isomorphism $M_2\in\sB(\cX_1,\cX_2)$ by $M_2 = U_0^{-1}R$, and therefore
\[ M = \begin{bmatrix} M_1 & 0\\ 0 & M_2\end{bmatrix}\colon \cX=  \ker(I_\cX-S_3T_3)\oplus \cX_1\to \ker U\oplus \cX_2 =\cX \]
is an isomorphism. For $x\in\ker(I_\cX-S_3T_3)$, we have $Mx\in\ker U$, so $UMx = 0 = (I_\cX-S_3T_3)x$; and for $x\in\cX_1$, we have $UMx= UU_0^{-1}Rx = (I_\cX-S_3T_3)x$. This shows that $UM = I_\cX -S_3T_3$ because $\cX =  \ker(I_\cX-T_3S_3) + \cX_1$.

\Cref{L:IndexEqual} implies that $\alpha(I_\cY-T_3S_3) = \alpha(I_\cX-S_3T_3) = m=\alpha(V)$, and
combining the identity
\[  I_\cY-T_3S_3 = B(I_\cY-T_2S_2)B^{-1} \]
with~\eqref{T:SCfredholmChar:eq1},
we deduce  that $\ran  (I_\cY-T_3S_3) =  \ran  V$.
Therefore we can repeat the constructions from the previous paragraphs to obtain an isomorphism $N\in\sB(\cY)$ such that $VN = I_\cY -T_3S_3$. Now the conclusion  that  $U$ and $V$ are SC follows from \Cref{char:SC}.

The implication \ref{T:SCfredholmChar2}$\Rightarrow$\ref{T:SCfredholmChar3} is clear.

\ref{T:SCfredholmChar3}$\Rightarrow$\ref{T:SCfredholmChar1}.  Suppose that $(U,V)\in\opSC_k(\cX,\cY)$. Then \Cref{char:SC} implies that there are operators $S\in\sB(\cY,\cX)$ and $T\in\sB(\cX,\cY)$  and  an isomorphism $M\in\sB(\cX)$ such that $UM=I_\cX-ST$. We have  $UM\in\Phi_k(\cX)$ because $U\in\Phi_k(\cX)$ and $M$ is an isomorphism, and consequently~\ref{T:SCfredholmChar1} is satisfied.
\end{proof}

\begin{proof}[Proof of Theorem~{\normalfont{\ref{T:SCdichotomy}}}]
\ref{T:SCdichotomy0} is simply a restatement of the equivalence of condi\-tions \ref{P:EAEfredholmIndex1} and~\ref{P:EAEfredholmIndex3} in \Cref{P:EAEfredholmIndex}.

  \ref{T:SCdichotomy2}. Take $k\in\eae(\cX,\cY)\Z$. By \ref{T:SCdichotomy0}, we have $\opEAE_k(\cX,\cY)\ne\emptyset$. In view of this, the implication~$\Rightarrow$ is clear, while the converse follows from \Cref{T:SCfredholmChar} (specifically, the implication \ref{T:SCfredholmChar3}$\Rightarrow$\ref{T:SCfredholmChar2}).

  \ref{T:SCdichotomy1}. \Cref{P:EAEfredholmIndex} shows that $\Phi_k(\cX)=\emptyset$ or $\Phi_k(\cY)=\emptyset$ for $k\in\Z\setminus\eae(\cX,\cY)\Z$. Since
\[ \opSC_k(\cX,\cY)\subseteq\opEAE_k(\cX,\cY)\subseteq\Phi_k(\cX)\times\Phi_k(\cY), \] we see  that $\opSC_k(\cX,\cY)=\opEAE_k(\cX,\cY)=\emptyset$ in this case.
\end{proof}

We conclude this section by showing how \Cref{T:SCfredholmChar} can be deduced from results obtained in~\cite{tHMR20}. To this end, take $(U,V)\in\opEAE_k(\cX,\cY)$. Translating the conclusion of \cite[Proposition~5.9]{tHMR20} about what is called ``the Banach space operator problem'' in~\cite{tHMR20} to the setting of EAE and SC, as explained in \cite[Section~3]{tHMR20}, we see that~$U$ and~$V$ are SC if and only if
there exist operators $B_1\in\sB(\ran U, \ran V)$ and $B_2\in\sB(\ran V, \ran U)$ such that
  \begin{equation}\label{Proof:alt}  I_{\ran V}- B_1B_2\in\Phi_k(\ran V).
\end{equation}
\begin{proof}[Alternative proof of Theorem~{\normalfont{\ref{T:SCfredholmChar}}}]
  As before, \Cref{L:IndexEqual} shows that conditions~\ref{T:SCfredholmChar1} and~\ref{T:SCfredholmChar4} are equivalent, and the implication \ref{T:SCfredholmChar2}$\Rightarrow$\ref{T:SCfredholmChar3} is trivial.

  \ref{T:SCfredholmChar4}$\Rightarrow$\ref{T:SCfredholmChar2}.  Suppose that~\ref{T:SCfredholmChar4} is satisfied, so that $I_\cY-TS\in\Phi_k(\cY)$ for some operators $S\in\sB(\cY,\cX)$ and $T\in\sB(\cX,\cY)$,   and take $(U,V)\in\opEAE_k(\cX,\cY)$. We must show that $(U,V)\in\opSC_k(\cX,\cY)$, which by the result from~\cite{tHMR20} stated above amounts to finding operators $B_1\in\sB(\ran U, \ran V)$ and $B_2\in\sB(\ran V, \ran U)$ which satisfy~\eqref{Proof:alt}.

  Take finite-dimensional subspaces~$\cX_1$ and~$\cY_1$ of~$\cX$ and~$\cY$, respectively, such that
\begin{equation}\label{eq:AltPf} \cX=\ran U \oplus \cX_1 \ands \cY=\ran V \oplus \cY_1, \end{equation} and decompose the operators~$S$ and~$T$ accordingly; that is,
\[
S=\mat{S_{11}&S_{12}\\ S_{21}&S_{22}}\quad\text{and}\quad
T=\mat{T_{11}&T_{12}\\ T_{21}&T_{22}}, \]
where $S_{11}\in\sB(\ran V,\ran U)$, $S_{12}\in\sB(\cY_1,\ran U)$, $S_{21}\in\sB(\ran V,\cX_1)$, $S_{22}\in\sB(\cY_1,\cX_1)$, $T_{11}\in\sB(\ran U,\ran V)$, $T_{12}\in\sB(\cX_1,\ran V)$, $T_{21}\in\sB(\ran U,\cY_1)$ and $T_{22}\in\sB(\cX_1,\cY_1)$.
Define
\[ S_1 = \mat{S_{11}&0\\ 0&0}\in\sB(\cY,\cX)\ands T_1  = \mat{T_{11}&0\\ 0&0}\in\sB(\cX,\cY). \]
Since $\cX_1$ and $\cY_1$ are finite-dimensional, the operators $S_2= S-S_1$ and $T_2= T-T_1$ have finite rank, and
\[
\Phi_k(\cY)\ni I_\cY -TS = I_\cY-T_1S_1 -(T_1S_2+T_2S_1+T_2S_2),
\]
where $T_1S_2+T_2S_1+T_2S_2$ is a finite-rank operator. Hence
\[ \Phi_k(\cY)\ni I_\cY-T_1S_1 =
\mat{I_{\ran V}-T_{11}S_{11}&0\\0& I_{\cY_1}},
\]
which in turn implies that $I_{\ran V}-T_{11}S_{11}\in\Phi_k(\ran V)$. Therefore the operators  $B_1 = T_{11}$ and $B_2 = S_{11}$  satisfy~\eqref{Proof:alt}.

\ref{T:SCfredholmChar3}$\Rightarrow$\ref{T:SCfredholmChar4}. Suppose that $\opSC_k(\cX,\cY)\ne\emptyset$, and take $(U,V)\in\opSC_k(\cX,\cY)$. Then, by the result from~\cite{tHMR20} stated above, we can find operators $B_1\in\sB(\ran U, \ran V)$ and $B_2\in\sB(\ran V, \ran U)$ which satisfy~\eqref{Proof:alt}. As before, take finite-dimensional subspaces~$\cX_1$ and~$\cY_1$ of~$\cX$ and~$\cY$, respectively, such that~\eqref{eq:AltPf}  is satisfied, and define
\[ S=\mat{B_2&0\\0&0}\colon \cY=\ran V\oplus\cY_1\to \ran U\oplus\cX_1=\cX \]
and
\[ T=\mat{B_1&0\\0&0}\colon\cX=\ran U\oplus\cX_1\to\ran V\oplus\cY_1=\cY. \]
Then
\[
I_{\cY}-TS=\mat{I_{\ran V}-B_1B_2 & 0\\ 0 & I_{\cY_1}} \in \Phi_k(\cY),
\]
which shows that~\ref{T:SCfredholmChar4} is satisfied.
\end{proof}

\section{Non-emptiness of $\opSC_k(\cX,\cY)$ and the proofs of Propositions~\ref{P:SC1806} and~\ref{P:beyondprojincomp}}\label{S:SC}
\noindent
The aim of this section is to investigate the set of integers~$k$ for which there \mbox{exist} Schur-coupled operators $U\in\Phi_k(\cX)$ and $V\in\Phi_k(\cY)$; that is, $\opSC_k(\cX,\cY)\ne\emptyset$.
We follow a similar strategy to the one successfully employed in \Cref{S:IntroProofs}, beginning with a partial analogue of  \Cref{L:indexideal} for the set $\mathbb{I}_{\text{SC}}(\cX,\cY) = \{ k\in\Z : \opSC_k(\cX,\cY)\ne\emptyset\}$ defined in~\eqref{eq:Isc}.
As we shall see, the situation for SC  is considerably more complicated than for EAE, primarily due to the difficulty of analyzing the technical conditions~\ref{T:SCfredholmChar1}--\ref{T:SCfredholmChar4} in \Cref{T:SCfredholmChar}.
In particular,
we have been unable to obtain an exact counterpart of  \Cref{L:indexideal} for SC because we do not know if the set $\mathbb{I}_{\text{SC}}(\cX,\cY)$ is always closed under addition.

\begin{lem}\label{L:SC}
  The set $\mathbb{I}_{\normalfont{\text{SC}}}(\cX,\cY)$ has the following properties   for every pair of Banach spaces~$(\cX,\cY)\colon$
\begin{enumerate}[label={\normalfont{(\roman*)}}]
\item\label{L:SC1} $0\in\mathbb{I}_{\normalfont{\text{SC}}}(\cX,\cY)$.

\item\label{L:SC3} $\operatorname{sc}(\cX,\cY)\in\mathbb{I}_{\normalfont{\text{SC}}}(\cX,\cY)$.
\item\label{L:SC2} $km\in\mathbb{I}_{\normalfont{\text{SC}}}(\cX,\cY)$ for every $k\in\mathbb{I}_{\normalfont{\text{SC}}}(\cX,\cY)$ and $m\in\Z$.

\item\label{L:SC5}
Suppose that
$\operatorname{sc}(\cX,\cY)\in\{0,\eae(\cX,\cY)\}$.
Then $\mathbb{I}_{\normalfont{\text{SC}}}(\cX,\cY) = \operatorname{sc}(\cX,\cY)\Z$.
\end{enumerate}
\end{lem}

\begin{proof} The first two properties are  easy to verify. Indeed, \ref{L:SC1}~follows from the fact that $(I_\cX,I_\cY)\in\opSC_0(\cX,\cY)$, while~\ref{L:SC3} follows from~\ref{L:SC1} if
  $\operatorname{sc}(\cX,\cY)=0$, and otherwise from the definition~\eqref{eq:sc} of  $\operatorname{sc}(\cX,\cY)$.

However, the proof of~\ref{L:SC2} requires more work.
  Take $k\in\mathbb{I}_{\normalfont{\text{SC}}}(\cX,\cY)$ and  $m\in\Z$. By~\ref{L:SC1}, we may suppose that  $k\ne 0$ and $m\ne 0$. Since $\opSC_k(\cX,\cY)\ne\emptyset$, \Cref{T:SCfredholmChar} implies that we can find operators $S_1\in\sB(\cY,\cX)$ and $T_1\in\sB(\cX,\cY)$ such that
  $I_\cX-S_1T_1\in\Phi_k(\cX)$.
  We claim that there exist  operators $S_m\in\sB(\cY,\cX)$ and $T_m\in\sB(\cX,\cY)$ such that
  \begin{equation}\label{L:SC:eq1} I_\cX-S_mT_m\in\Phi_{km}(\cX). \end{equation}
  Once we have  established this claim, the conclusion will follow from another application of \Cref{T:SCfredholmChar}.

  We prove the claim by considering three different cases: $m\ge 2$, $m=-1$ and $m\le-2$. (Note that the case $m=1$ is already covered by the choice of~$S_1$ and~$T_1$.)

  \emph{Case 1.} For $m\ge2$, we can apply the Binomial Theorem because~$I_\cX$ and~$S_1T_1$ commute. It shows that
  \[  (I_\cX-S_1T_1)^m = I_\cX + \sum_{j=1}^m\binom{m}{j} (-S_1T_1)^j = I_\cX -S_1T_1\sum_{j=1}^m\binom{m}{j} (-S_1T_1)^{j-1}, \]
  so the Index Theorem implies that the operators $S_m = S_1\in\sB(\cY,\cX)$ and $T_m= T_1\sum_{j=1}^m\binom{m}{j}(-S_1T_1)^{j-1}\in\sB(\cX,\cY)$ satisfy~\eqref{L:SC:eq1}.

      \emph{Case 2.} For $m=-1$, we consider the cases $k>0$ and $k<0$ separately. For \mbox{$k>0$}, \Cref{perturbkernel2} implies that we can find operators $U\in\sB(\cY,\cX)$ and $V\in\sB(\cX,\cY)$ such that $S_1T_1-UV$ is a finite-rank operator and $I_\cX-UV\in\Phi_k(\cX)$ with $\alpha(I_\cX-UV)=k$. Then $\beta(I_\cX-UV)=0$, so $I_\cX-UV$ is a surjective Fredholm operator, and therefore it has a right inverse~$R\in\Phi_{-k}(\cX)$.  Consequently
      \[ I_\cX = (I_\cX-UV)R = R - UVR, \]
which implies that $S_{-1} = -U\in\sB(\cY,\cX)$ and $T_{-1} = VR\in\sB(\cX,\cY)$ satisfy~\eqref{L:SC:eq1} because $I_\cX-S_{-1}T_{-1} = R\in\Phi_{-k}(\cX)$.

      The argument for $k<0$ is very similar. In this case, we can  apply \Cref{perturbkernel2} to find  $U\in\sB(\cY,\cX)$ and $V\in\sB(\cX,\cY)$ such that $S_1T_1-UV$ is a finite-rank operator and $I_\cX-UV\in\Phi_k(\cX)$ with $\alpha(I_\cX-UV)=0$. Then, being  an injective Fredholm operator, $I_\cX-UV$ has a left inverse~$L\in\Phi_{-k}(\cX)$, which implies that the operators $S_{-1} = -LU\in\sB(\cY,\cX)$ and $T_{-1} = V\in\sB(\cX,\cY)$ satisfy
      \[ I_\cX-S_{-1}T_{-1} = L(I_\cX-UV) + (LU)V = L\in\Phi_{-k}(\cX). \]

      \emph{Case 3.}   Finally, for $m\le -2$, we apply the argument from Case 1 to the $-m^{\text{th}}$ power of the operator $I_\cX-S_{-1}T_{-1}$, where $S_{-1}$ and $T_{-1}$ are the operators found in Case 2, to conclude that the  operators $S_m=S_{-1}\in\sB(\cY,\cX)$ and $T_m= T_{-1} \sum_{j=1}^{-m}\binom{-m}{j}(-S_{-1}T_{-1})^{j-1}\in\sB(\cX,\cY)$ satisfy~\eqref{L:SC:eq1}.

\ref{L:SC5}. First, if $\operatorname{sc}(\cX,\cY) =0$,  then the definition~\eqref{eq:sc} of~$\operatorname{sc}(\cX,\cY)$ implies that $\mathbb{I}_{\normalfont{\text{SC}}}(\cX,\cY)\cap\N =\emptyset$, and therefore  $\mathbb{I}_{\normalfont{\text{SC}}}(\cX,\cY)=\{0\}=\operatorname{sc}(\cX,\cY)\Z$ by~\ref{L:SC1} and~\ref{L:SC2}.

Second, suppose that $\eae(\cX,\cY) = \operatorname{sc}(\cX,\cY)$, and call this number $k_0$. Then, by~\ref{L:SC3}, \ref{L:SC2} and  \Cref{C:IphiXcapIphiY}, we have
  \[ k_0\Z\subseteq  \mathbb{I}_{\normalfont{\text{SC}}}(\cX,\cY)\subseteq \mathbb{I}_\Phi(\cX)\cap\mathbb{I}_\Phi(\cY) = k_0\Z, \]
which shows that $\mathbb{I}_{\normalfont{\text{SC}}}(\cX,\cY)=k_0\Z$, as required.
\end{proof}

\begin{proof}[Proof of Proposition~{\normalfont{\ref{P:SC1806}}}]
  We begin by verifying the chain of inclusions~\eqref{SC1806incl}, which we restate here for ease of reference:
    \begin{align*}
\operatorname{sc}(\cX,\cY)\Z &\subseteq  \mathbb{I}_{\normalfont{\text{SC}}}(\cX,\cY)
=  \{ k\in\Z : \opSC_k(\cX,\cY)
= \opEAE_k(\cX,\cY)\ne\emptyset\}\\
&\subseteq\eae(\cX,\cY)\Z
=  \{ k\in\Z : \opEAE_k(\cX,\cY)\ne\emptyset\}
=  \mathbb{I}_{\Phi}(\cX)\cap\mathbb{I}_{\Phi}(\cY).
\end{align*}
    The first inclusion    is immediate from \Cref{L:SC}\ref{L:SC3}--\ref{L:SC2}, while the equality of the second and  third set in the first line follows by  combining the definition of~$\mathbb{I}_{\normalfont{\text{SC}}}(\cX,\cY)$ with the equivalence of conditions~\ref{T:SCfredholmChar2} and~\ref{T:SCfredholmChar3} in \Cref{T:SCfredholmChar}. \Cref{P:EAEfredholmIndex} shows that the three sets in the second line are equal (the equality of the first and last of these sets was also recorded in  \Cref{C:IphiXcapIphiY}), and finally the inclusion at the beginning of the second line follows because the final set in the first line is trivially contained in the second set in the second line.

Next, to prove the first claim of  Proposition~\ref{P:SC1806},  suppose that $\opEAE_k(\cX,\cY)=\opSC_k(\cX,\cY)$ for every $k\in\Z$. Then
\[ \mathbb{I}_{\normalfont{\text{SC}}}(\cX,\cY) = \{ k\in\Z : \opEAE_k(\cX,\cY)\ne\emptyset\}  = \mathbb{I}_\Phi(\cX)\cap\mathbb{I}_\Phi(\cY), \]
where the final equality follows from~\eqref{SC1806incl}. Hence the definitions~\eqref{eq:eaeDefn} and~\eqref{eq:sc} show that $\eae(\cX,\cY) = \operatorname{sc}(\cX,\cY)$.

Conversely, suppose that $\eae(\cX,\cY) = \operatorname{sc}(\cX,\cY)$, and call this number $k_0$. Then the inclusions in~\eqref{SC1806incl} are in fact equalities, so \mbox{$\opEAE_k(\cX,\cY)=\opSC_k(\cX,\cY)\, (\ne\emptyset)$} for every $k\in\mathbb{I}_{\normalfont{\text{SC}}}(\cX,\cY) = k_0\Z$. On the other hand, \Cref{T:SCdichotomy}\ref{T:SCdichotomy1} shows that the identity $\opEAE_k(\cX,\cY)=\opSC_k(\cX,\cY)\,  (=\emptyset)$ is also true for every $k\notin\eae(\cX,\cY)\Z= k_0\Z$.

Finally, we verify that  $\operatorname{sc}(\cX,\cY) = n\eae(\cX,\cY)$  for some $n\in\N_0$. Set  $k_0=\operatorname{sc}(\cX,\cY)\in\N_0$. Then \Cref{L:SC}\ref{L:SC3} shows that $\emptyset\ne \opSC_{k_0}(\cX,\cY)\subseteq\opEAE_{k_0}(\cX,\cY)$, so  $k_0=n\eae(\cX,\cY)$ for some $n\in\N_0$ by \Cref{P:EAEfredholmIndex} and the fact that~$k_0$ and $\eae(\cX,\cY)$ are both non-negative.
\end{proof}

  \begin{rem} To illustrate the applicability of our work thus far, let us explain how it leads to an explicit algorithm for deciding whether EAE and SC are equivalent for all pairs of Fredholm operators on a given pair of Banach spaces~$(\cX,\cY)$.
\begin{enumerate}[label={\normalfont{(\roman*)}}]
\item Find, if possible, the least $k\in\N$ such that $\Phi_k(\cX)\ne\emptyset$ and $\Phi_k(\cY)\ne \emptyset$. This is $k_0=\eae(\cX,\cY)$.
\item If no such~$k\in\N$ exists, then  $\eae(\cX,\cY)=0 =\operatorname{sc}(\cX,\cY)$, and EAE and SC are equivalent for all  pairs of Fredholm operators on~$\cX$ and~$\cY$ by \Cref{P:SC1806}. More precisely, we have
  \[ \opSC_k(\cX,\cY)=\opEAE_k(\cX,\cY)\ \begin{cases} \ne\emptyset\quad &\text{for}\  k=0\\
  =\emptyset\quad &\text{for}\ k\in\Z\setminus\{0\}. \end{cases} \]
\item Otherwise choose any pair of Fredholm operators $(U,V)\in\Phi_{k_0}(\cX)\times\Phi_{k_0}(\cY)$ with $\alpha(U)=\alpha(V)$ (and hence $\beta(U)=\beta(V)$),
  and decide whether~$U$ and~$V$ are SC.
 \begin{enumerate}[label={\normalfont{(\arabic*)}}]
\item If $U$ and $V$ are SC, then \Cref{P:SC1806} implies that  EAE and SC are equivalent for all pairs of Fredholm operators on~$\cX$ and~$\cY$,  and
\[ \opSC_k(\cX,\cY)=\opEAE_k(\cX,\cY)\ \begin{cases} \ne\emptyset\quad &\text{for}\  k\in k_0\Z\\ =\emptyset\quad &\text{for}\ k\in\Z\setminus k_0\Z. \end{cases} \]
\item Otherwise   EAE and SC are evidently not equivalent for all pairs of Fredholm operators  on~$\cX$ and~$\cY$,   as $(U,V)$ is a concrete example of a pair which is EAE, but not SC.
\end{enumerate}
\end{enumerate}
\end{rem}

  Our next lemma  is the counterpart of \Cref{L:isoInvarianceIphi} for SC, showing that  the set  $\mathbb{I}_{\normalfont{\text{SC}}}(\cX,\cY)$, and hence the associated index  $\operatorname{sc}(\cX,\cY)$, is an isomorphic invariant.

\begin{lem}\label{L:isoInvarianceISC}
  Let $\cX_1$, $\cX_2$, $\cY_1$ and~$\cY_2$ be Banach spaces satisfying $\cX_1\cong\cX_2$ and \mbox{$\cY_1\cong \cY_2$}. Then $\mathbb{I}_{\normalfont{\text{SC}}}(\cX_1,\cY_1) = \mathbb{I}_{\normalfont{\text{SC}}}(\cX_2,\cY_2)$, and
consequently $\operatorname{sc}(\cX_1,\cY_1) = \operatorname{sc}(\cX_2,\cY_2)$.
\end{lem}

\begin{proof} Take  isomorphisms $R\in\sB(\cX_1,\cX_2)$ and $S\in\sB(\cY_1,\cY_2)$. For each $k\in\mathbb{I}_{\text{SC}}(\cX_1,\cY_1)$, we can find $U\in\Phi_k(\cX_1)$ and $V\in\Phi_k(\cY_1)$ which are SC, so that there exist  isomorphisms $A\in\sB(\cX_1)$ and $D\in\sB(\cY_1)$ and operators $B\in\sB(\cY_1,\cX_1)$ and $C\in\sB(\cX_1,\cY_1)$  such that~\eqref{SC} is satisfied. The Index Theorem implies that $RUR^{-1}\in\Phi_k(\cX_2)$ and $SVS^{-1}\in\Phi_k(\cY_2)$, and it is easy to check that they are SC, using  the operators  $RAR^{-1}\in\sB(\cX_2)$, $SDS^{-1}\in\sB(\cY_2)$, $RBS^{-1}\in\sB(\cY_2,\cX_2)$ and $SCR^{-1}\in\sB(\cX_2,\cY_2)$ to verify~\eqref{SC}.
  This implies that  \mbox{$k\in\mathbb{I}_{\text{SC}}(\cX_2,\cY_2)$}, so $\mathbb{I}_{\normalfont{\text{SC}}}(\cX_1,\cY_1)\subseteq\mathbb{I}_{\normalfont{\text{SC}}}(\cX_2,\cY_2)$.  The opposite inclusion follows by interchanging~$\cX_1$ and~$\cX_2$, and~$\cY_1$ and~$\cY_2$.

The  final statement is immediate from  the definition~\eqref{eq:sc} of~$\operatorname{sc}$.
\end{proof}

As another  consequence of \Cref{L:SC}, we obtain   the following variant of \Cref{P:EAEfredholmIndex} for SC.

\begin{cor}\label{P:SC} Let~$\cX$ and~$\cY$ be Banach spaces. Then  the set $\mathbb{I}_{\normalfont{\text{SC}}}(\cX,\cY)$ is closed under addition if and only if
$\mathbb{I}_{\normalfont{\text{SC}}}(\cX,\cY) = \operatorname{sc}(\cX,\cY)\Z$.
\end{cor}

\begin{proof} The implication $\Leftarrow$ is obvious because  the set $\operatorname{sc}(\cX,\cY)\Z$  is closed under addition. Conversely,  suppose that $\mathbb{I}_{\text{SC}}(\cX,\cY)$ is closed under addition. Then, in view of \Cref{L:SC}\ref{L:SC1} and~\ref{L:SC2}, it is an ideal of~$\Z$, so  $\mathbb{I}_{\text{SC}}(\cX,\cY)=m\Z$ for some $m\in\N_0$.  Combining this identity with the definition~\eqref{eq:sc} of $\operatorname{sc}(\cX,\cY)$, we conclude that $m=\operatorname{sc}(\cX,\cY)$.
  \end{proof}

\Cref{P:SC} is not entirely satisfactory because we have been unable to answer the following question.

\begin{quest}\label{Q:IscAdditive} Is the set $\mathbb{I}_{\normalfont{\text{SC}}}(\cX,\cY)$  closed under addition for every  pair of Banach spaces~$(\cX,\cY)$?
\end{quest}

 We know that the answer to this question is ``yes'' in certain cases because \Cref{L:SC}\ref{L:SC5} shows that $\mathbb{I}_{\normalfont{\text{SC}}}(\cX,\cY)$  is closed under addition if $\operatorname{sc}(\cX,\cY) =0$ or $\operatorname{sc}(\cX,\cY) = \eae(\cX,\cY)$. We can also obtain a positive answer to it by imposing suitable conditions on the Banach spaces~$\cX$ and~$\cY$.
To state this result precisely, we require the following additional notation and terminology.

\begin{defn} Let~$\cX$ and~$\cY$ be Banach spaces.
 \begin{enumerate}[label={\normalfont{(\roman*)}}]
  \item Set $\sG_\cY(\cX) =\{ ST : S\in\sB(\cY,\cX),\, T\in\sB(\cX,\cY)\}$.
  \item We say that a subset~$\Sigma$ of~$\sB(\cX,\cY)$ is \emph{essentially closed under addition} if, for every pair of operators $U,V\in\Sigma$,  there exists an inessential operator  $R\in\sE(\cX,\cY)$  such that $U+V-R\in\Sigma$.
\end{enumerate}
\end{defn}

\begin{prop}\label{P:IscAdditive}
  Let $\cX$ and $\cY$ be Banach spaces, and suppose
that at least one of the sets~$\sG_\cY(\cX)$ and $\sG_\cX(\cY)$ is essentially closed under addition.
Then~$\mathbb{I}_{\normalfont{\text{SC}}}(\cX,\cY)$ is closed under addition.
\end{prop}

\begin{proof}
Suppose that~$\sG_\cY(\cX)$ is essentially closed under addition, and take $k_1,k_2\in\mathbb{I}_{\normalfont{\text{SC}}}(\cX,\cY)$. Then by \Cref{T:SCfredholmChar}, we can find operators $S_j\in\sB(\cY,\cX)$ and $T_j\in\sB(\cX,\cY)$ such that $I_\cX - S_jT_j\in\Phi_{k_j}(\cX)$  for $j=1,2$. The Index Theorem shows that
\begin{equation}\label{P:IscAdditive:eq} \Phi_{k_1+k_2}(\cX)\ni(I_\cX-S_1T_1)(I_\cX-S_2T_2) = I_\cX - [S_1T_1(I_\cX-S_2T_2) + S_2T_2]. \end{equation}
Both of the operators $S_1T_1(I_\cX-S_2T_2)$ and $S_2T_2$ belong to $\sG_\cY(\cX)$, so by
the hypothesis, we can find operators $S_3\in\sB(\cY,\cX)$, $T_3\in\sB(\cX,\cY)$ and $R\in\sE(\cX)$ such that $S_1T_1(I_\cX-S_2T_2) + S_2T_2 = S_3T_3+R$. Combining~\eqref{P:IscAdditive:eq} with \Cref{R:ss_iness}\ref{inessperturb}, we deduce that
\[ \Phi_{k_1+k_2}(\cX)\ni I_\cX - [S_1T_1(I_\cX-S_2T_2) + S_2T_2]+R =
I_\cX-S_3T_3, \]
 and therefore $\opSC_{k_1+k_2}(\cX,\cY)\ne\emptyset$ by another application of \Cref{T:SCfredholmChar}. This shows that $k_1+k_2\in\mathbb{I}_{\normalfont{\text{SC}}}(\cX,\cY)$, as required.

The case where~$\sG_\cX(\cY)$ is essentially closed under addition is similar, just using condition~\ref{T:SCfredholmChar4} in \Cref{T:SCfredholmChar} instead of condition~\ref{T:SCfredholmChar1}.
\end{proof}

\begin{rem}\label{R:Aug2022}
  The set~$\sG_\cY(\cX)$ is closed under addition (without the need for any inessential perturbations) if the Banach space~$\cY$ contains a complemented subspace isomorphic to~\mbox{$\cY\oplus \cY$}. This result is ``folklore''; it can for instance be found in \cite[the paragraph following Definition~3.6]{LLR}. Most ``classical'' Banach spaces~$\cY$ satisfy the even stronger condition that $\cY\cong \cY\oplus\cY$. The two conditions are not equivalent because Gowers and Maurey~\cite[\S{}(4.4)]{gm2} have constructed a Banach space~$\cY$ which is isomorphic to its cube~$\cY\oplus\cY\oplus\cY$, but not to its square~$\cY\oplus\cY$. Hence~$\cY$ contains a complemented subspace isomorphic to~$\cY\oplus \cY$ without being isomorphic to it.

  There are infinite-dimensional Banach spaces~$\cY$ which do not contain any complemented subspaces isomorphic to~\mbox{$\cY\oplus \cY$}. James' quasi-reflexive Banach space, which will feature prominently in the next example, was the first such example.
\end{rem}

\begin{example}\label{Ex:James} The purpose of this example is to show that the converse of \Cref{P:IscAdditive} fails; that is, we shall construct Banach spaces~$\cX$ and~$\cY$ such that~$\mathbb{I}_{\normalfont{\text{SC}}}(\cX,\cY)$ is closed under addition, but neither~$\sG_\cY(\cX)$ nor~$\sG_\cX(\cY)$ are essentially closed under addition.
  This construction relies heavily on the quasi-reflexive James spaces~$\cJ_p$ for $1<p<\infty$. These Banach spaces originate in James' paper~\cite{Ja}, where only the case $p=2$ was considered. Subsequently, Edelstein and Mityagin~\cite{em} observed that James' methods and results carry over to arbitrary $p\in(1,\infty)$.
  We require the following specific facts about this family of Banach spaces:
\begin{enumerate}[label={\normalfont{(\roman*)}}]
\item\label{Ex:James1} $\cJ_p$ is isomorphic to its hyperplanes for every  $p\in(1,\infty)$, so $\gamma(\cJ_p)=1$.
  \item\label{Ex:James2} $\sB(\cJ_q,\cJ_p)=\sK(\cJ_q,\cJ_p)$ for $1<p<q<\infty$ by \cite[Theorem~4.5]{LW}, and therefore~$\cJ_p$ and~$\cJ_q$ are essentially incomparable whenever $p,q\in(1,\infty)$ are distinct.
  \item\label{Ex:James3} $\sK(\cJ_p)=\sE(\cJ_p)\subseteq\sW(\cJ_p)$ for every $p\in(1,\infty)$ by \cite[Proposition~4.9]{NJLmax}, where~$\sW(\cJ_p)$ denotes the ideal of weakly compact operators on~$\cJ_p$.

Berkson and Porta \cite[page~18]{bp} for $p=2$ and Edel\-stein and Mit\-ya\-gin~\cite{em} for general~$p\in(1,\infty)$ observed that~$\sW(\cJ_p)$ has codimension~$1$ in~$\sB(\cJ_p)$, so we have a unital algebra homomorphism $\varphi\colon\sB(\cJ_p)\to\mathbb{K}$ with \mbox{$\ker\varphi = \sW(\cJ_p)$}. We shall in fact require the amplification of this homomorphism to the $2\times 2$ matrices, that is, the unital algebra homomorphism \mbox{$\varphi_2\colon M_2(\sB(\cJ_p))\to M_2(\mathbb{K})$} given by
\[ \varphi_2\left(\mat{R_{11} & R_{12}\\ R_{21} & R_{22}} \right)
 = \mat{ \varphi(R_{11}) & \varphi(R_{12})\\ \varphi(R_{21}) & \varphi(R_{22})}.
 \]
 \end{enumerate}
We are now ready to begin our construction: Choose  distinct numbers $p,q\in(1,\infty)$, and set $\cX=\cJ_p\oplus \cJ_p\oplus \cJ_q$ and $\cY=\cJ_p\oplus \cJ_q\oplus \cJ_q$.

First, we observe that~$\mathbb{I}_{\text{SC}}(\cX,\cY)$ is  closed under addition. This follows immediately from the fact that $\mathbb{I}_{\text{SC}}(\cX,\cY)=\Z$. Indeed, for each $k\in\Z$, \ref{Ex:James1} implies that we can take $R\in\Phi_k(\cJ_p)$. Then the operators
\[ S =\mat{I_{\cJ_p} - R & 0 & 0\\ 0 & 0 & 0\\ 0 & 0 & 0}
\in\sB(\cY,\cX)\qquad\text{and}\qquad T = \mat{I_{\cJ_p} & 0 & 0\\ 0 & 0 & 0\\ 0 & 0 & 0}
\in\sB(\cX,\cY) \]
    satisfy
\[ I_{\cX} - ST = \mat{R & 0 & 0\\ 0 & I_{\cJ_p} & 0\\ 0 & 0 & I_{\cJ_q}}
\in\Phi_k(\cX), \]
so $\opSC_k(\cX,\cY)\ne\emptyset$ by \Cref{T:SCfredholmChar}, and therefore $k\in\mathbb{I}_{\text{SC}}(\cX,\cY)$, as desired.

Second, we shall show the set~$\sG_\cY(\cX)$
 is not essentially closed under addition. Assume the contrary, and consider the operators
\begin{align*} U &=\mat{I_{\cJ_p} & 0 & 0\\ 0 & 0 & 0\\ 0 & 0 & 0} = \mat{I_{\cJ_p} & 0 & 0\\ 0 & 0 & 0\\ 0 & 0 & 0}  \mat{I_{\cJ_p} & 0 & 0\\ 0 & 0 & 0\\ 0 & 0 & 0}\in\sB(\cX)\\ \intertext{and}
  V &= \mat{0 & 0 & 0\\ 0 & I_{\cJ_p} & 0\\ 0 & 0 & 0} = \mat{0 & 0 & 0\\ I_{\cJ_p}  & 0 & 0\\ 0 & 0 & 0} \mat{0 & I_{\cJ_p} & 0\\ 0 & 0 & 0\\ 0 & 0 & 0}\in\sB(\cX). \end{align*} They both belong to~$\sG_{\cY}(\cX)$ as the indicated factorizations show.
Therefore, by the hypo\-thesis, we can find operators $S\in\sB(\cY,\cX)$ and $T\in\sB(\cX,\cY)$ such that \mbox{$U+V - ST\in\sE(\cX)$}. By~\ref{Ex:James2}, we can write $S = S_1+S_2$ and $T=T_1+T_2$, where
  \[ S_1 = \mat{S_{11} & 0 & 0\\ S_{21} & 0 & 0\\ 0 & S_{32} & S_{33}} \qquad\text{and}\qquad
T_1 =\mat{T_{11} & T_{12} & 0\\ 0 & 0 & T_{23}\\ 0 & 0 & T_{33}}, \]
and $S_2\in\sE(\cY,\cX)$ and $T_2\in\sE(\cX,\cY)$. Since~$\sE$ is an operator ideal, we deduce that
\begin{multline*} \sE(\cJ_p\oplus \cJ_p)\ni \mat{I_{\cJ_p} & 0 & 0\\ 0 & I_{\cJ_p} & 0}
(U+V-S_1T_1) \mat{I_{\cJ_p} & 0\\ 0 & I_{\cJ_p}\\ 0 & 0}\\ =  \mat{I_{\cJ_p} & 0 \\ 0 & I_{\cJ_p}} - \mat{S_{11}T_{11} & S_{11}T_{12}\\ S_{21}T_{11} & S_{21}T_{12}}.
\end{multline*}
Hence, applying the algebra homomorphism~$\varphi_2$ from~\ref{Ex:James3}, we obtain
\begin{align*}
\mat{1 & 0\\ 0 & 1}
 = \varphi_2 \left(\mat{I_{\cJ_p} & 0 \\ 0 & I_{\cJ_p}} \right)
 &= \varphi_2 \left(\mat{S_{11}T_{11} & S_{11}T_{12}\\ S_{21}T_{11} & S_{21}T_{12}} \right)\\
 &= \mat{ \varphi(S_{11})\varphi(T_{11}) & \varphi(S_{11})\varphi(T_{12})\\ \varphi(S_{21})\varphi(T_{11}) & \varphi(S_{21})\varphi(T_{12})}. \end{align*}
  However, this is impossible because the diagonal entries imply that $\varphi(S_{11})$, $\varphi(T_{11})$, $\varphi(S_{21})$ and~$\varphi(T_{12})$ are all non-zero, but then the off-diagonal entries $\varphi(S_{11})\varphi(T_{12})$ and $\varphi(S_{21})\varphi(T_{11})$ are also non-zero. This contradiction proves that~$\sG_\cX(\cY)$ cannot be essentially closed under addition.

Finally, a similar argument with~$\cX$ and~$\cY$  interchanged shows that the set~$\sG_\cX(\cY)$ is not essentially closed under addition.
\end{example}

\begin{rem}\label{C:SC=EAE} The purpose of this remark is to summarize our knowledge  about the values of $k\in\Z$ for which the equation $\opSC_k(\cX,\cY) = \opEAE_k(\cX,\cY)$ holds true for a given pair of  Banach spaces~$(\cX,\cY)$, and explain how this problem is related to  \Cref{Q:IscAdditive}.  Recall from~\eqref{SC1806incl} and~\Cref{T:SCdichotomy}\ref{T:SCdichotomy1} that
\begin{equation}\label{C:SC=EAE:eq1}
  \begin{alignedat}{2} \opSC_k(\cX,\cY)&=\opEAE_k(\cX,\cY)\ne\emptyset\quad &&\text{for}\quad k\in\mathbb{I}_{\normalfont{\text{SC}}}(\cX,\cY),\\
   \opSC_k(\cX,\cY)&=\opEAE_k(\cX,\cY)=\emptyset\quad &&\text{for}\quad k\in\Z\setminus\eae(\cX,\cY)\Z.
  \end{alignedat}
\end{equation}

We now split in two cases, beginning with the case where the answer to \Cref{Q:IscAdditive} is affirmative, so that the set $\mathbb{I}_{\normalfont{\text{SC}}}(\cX,\cY)$ is closed under addition.  Then \Cref{P:SC} shows that $\mathbb{I}_{\normalfont{\text{SC}}}(\cX,\cY)=\operatorname{sc}(\cX,\cY)\Z$ and
  \[     \emptyset=\opSC_k(\cX,\cY)\ne\opEAE_k(\cX,\cY)\quad \text{for}\quad k\in\eae(\cX,\cY)\Z\setminus\mathbb{I}_{\normalfont{\text{SC}}}(\cX,\cY), \]
  where we have applied \Cref{P:EAEfredholmIndex} to conclude that $\opEAE_k(\cX,\cY)\ne\emptyset$.
  Together with~\eqref{C:SC=EAE:eq1}, this covers all  possible values of~$k\in\Z$. It follows in particular that EAE and SC coincides for all pairs of Fredholm operators $(U,V)\in\Phi(\cX)\times\Phi(\cY)$ if and only if $\eae(\cX,\cY)=\operatorname{sc}(\cX,\cY)$, as we have already seen in \Cref{P:SC1806}.

Otherwise, when the answer to \Cref{Q:IscAdditive} is negative, so that  $\mathbb{I}_{\normalfont{\text{SC}}}(\cX,\cY)$ fails to be closed under addi\-tion, \Cref{P:SC1806} implies that $\operatorname{sc}(\cX,\cY)=n\eae(\cX,\cY)$ for some $n\ge 2$, and   \[     \emptyset=\opSC_k(\cX,\cY)\ne\opEAE_k(\cX,\cY) \]
  for every $k\in\{\pm m\eae(\cX,\cY) : 1\le m <n\}$ because $k=\operatorname{sc}(\cX,\cY)$ is the smallest positive number for which $\opSC_k(\cX,\cY)\ne\emptyset$, and $\mathbb{I}_{\text{SC}}(\cX,\cY)$ is closed under sign changes. However, since $\mathbb{I}_{\normalfont{\text{SC}}}(\cX,\cY)$ fails to be closed under addi\-tion, there must be some number $k=m\eae(\cX,\cY)$, where  $m\in\N\cap(n,\infty)\setminus n\N$, for which \[ \opSC_k(\cX,\cY) =\opEAE_k(\cX,\cY)\ne\emptyset. \]
\end{rem}

We conclude this section with a couple of  results about direct sums which will be useful in the proof of \Cref{P:beyondprojincomp}, as well as in the next section.

\begin{lem}\label{P:SCcomplSubspace} Let~$\cX$ and~$\cY$ be Banach spaces, and suppose that~$\cY$ is isomorphic to a complemented subspace of~$\cX$. Then
\begin{equation}\label{P:SCcomplSubspace:eq1}
  \mathbb{I}_{\normalfont{\text{SC}}}(\cX,\cY)=\mathbb{I}_\Phi(\cY)\subseteq \mathbb{I}_\Phi(\cX),
\end{equation}
and consequently $\operatorname{sc}(\cX,\cY)=\eae(\cX,\cY)$.
\end{lem}

\begin{proof}
  In view of Lemmas~\ref{L:isoInvarianceIphi} and~\ref{L:isoInvarianceISC}, we may suppose that $\cX=\cY\oplus\cZ$ for some Banach space~$\cZ$.

  The inclusion $\mathbb{I}_{\normalfont{\text{SC}}}(\cX,\cY)\subseteq\mathbb{I}_\Phi(\cY)$ is clear because $\opSC_k(\cX,\cY)\subseteq\Phi_k(\cX)\times\Phi_k(\cY)$.

  Conversely, for $k\in\mathbb{I}_\Phi(\cY)$, we can take $R\in\Phi_k(\cY)$. Then the operators
  \[ S = \begin{bmatrix} I_\cY-R\\ 0\end{bmatrix}\colon\ \cY\to\cY\oplus\cZ =\cX\qquad\text{and}\qquad T = \begin{bmatrix} I_\cY & 0\end{bmatrix}\colon\ \cX =\cY\oplus\cZ\to\cY \]
satisfy
\begin{equation}\label{P:SCcomplSubspace:eq2} I_\cX -ST = \begin{bmatrix} I_\cY & 0\\ 0 & I_\cZ\end{bmatrix} - \begin{bmatrix} I_\cY-R & 0\\ 0 & 0\end{bmatrix} = \begin{bmatrix} R & 0\\ 0 & I_\cZ\end{bmatrix}\in\Phi_k(\cX), \end{equation}
so $k\in\mathbb{I}_{\normalfont{\text{SC}}}(\cX,\cY)$ by \Cref{T:SCfredholmChar}. This shows that $ \mathbb{I}_{\normalfont{\text{SC}}}(\cX,\cY)=\mathbb{I}_\Phi(\cY)$, while
 the inclusion $\mathbb{I}_\Phi(\cY)\subseteq \mathbb{I}_\Phi(\cX)$ is an immediate consequence of~\eqref{P:SCcomplSubspace:eq2}. (Alternatively, the latter inclusion follows easily from \Cref{L:essincFredholmIndex}\ref{L:essincFredholmIndex1}.)

Finally, we have $\operatorname{sc}(\cX,\cY)=\eae(\cX,\cY)$ because~\eqref{P:SCcomplSubspace:eq1} shows that
$\mathbb{I}_{\normalfont{\text{SC}}}(\cX,\cY)=\mathbb{I}_\Phi(\cX)\cap\mathbb{I}_\Phi(\cY)$.
\end{proof}

\begin{lem}\label{L:essincomp}
  Let $\cX=\cX_1\oplus\cX_2$ and $\cY=\cY_1\oplus\cY_2$ be Banach spaces. Then:
\begin{enumerate}[label={\normalfont{(\roman*)}}]
\item\label{L:essincomp1}
  $\mathbb{I}_{\normalfont{\text{SC}}}(\cX_2,\cY_2)\subseteq \mathbb{I}_{\normalfont{\text{SC}}}(\cX,\cY)$.
\item\label{L:essincomp2}
  Suppose that each of the pairs $(\cX_1,\cY_1)$, $(\cX_1,\cY_2)$ and $(\cY_1,\cX_2)$ is essentially incomparable. Then   $\mathbb{I}_{\normalfont{\text{SC}}}(\cX,\cY)=
  \mathbb{I}_{\normalfont{\text{SC}}}(\cX_2,\cY_2)$.
\end{enumerate}
\end{lem}

\begin{proof}
  \ref{L:essincomp1}.   For each $k\in\mathbb{I}_{\normalfont{\text{SC}}}(\cX_2,\cY_2)$, we can take Schur-coupled operators $U\in\Phi_k(\cX_2)$ and  $V\in\Phi_k(\cY_2)$. Choose
  isomorphisms $A\in\sB(\cX_2)$ and $D\in\sB(\cY_2)$ and operators $B\in\sB(\cY_2,\cX_2)$ and $C\in\sB(\cX_2,\cY_2)$  such that~\eqref{SC} is satisfied. Then it is easy to see that the operators
  \[
  \mat{ I_{\cX_1} &0\\ 0& U}\in\Phi_k(\cX)\qquad \text{and}\qquad  \mat{ I_{\cY_1} &0\\ 0& V}\in\Phi_k(\cY) \]
      are Schur-coupled via
\begin{align*}
\mat{ I_{\cX_1} &0\\ 0& U} &=  \mat{ I_{\cX_1} &0\\ 0& A} - \mat{ 0 & 0\\ 0& B } \mat{ I_{\cY_1} &0\\ 0& D^{-1}} \mat{ 0 & 0\\ 0 & C}\\
\intertext{and}
\mat{ I_{\cY_1} &0\\ 0& V} &=  \mat{ I_{\cY_1} &0\\ 0& D } - \mat{ 0 & 0\\ 0& C } \mat{ I_{\cX_1} &0\\ 0& A^{-1} } \mat{ 0 & 0\\ 0 & B }, \end{align*}
so we conclude that $k\in\mathbb{I}_{\normalfont{\text{SC}}}(\cX,\cY)$.

\ref{L:essincomp2}. Suppose that  $k\in\mathbb{I}_{\normalfont{\text{SC}}}(\cX,\cY)$. By \Cref{T:SCfredholmChar}, we can find operators
\[ S = \begin{bmatrix} S_{11} &  S_{12}\\ S_{21} & S_{22} \end{bmatrix}\in\sB(\cY,\cX)\qquad\text{and}\qquad T= \begin{bmatrix}  T_{11} &  T_{12}\\ T_{21}& T_{22}\end{bmatrix}\in\sB(\cX,\cY) \]
such that $I_{\cX}-ST\in\Phi_k(\cX)$.  The hypothesis implies that~$S_{11}$, $T_{11}$, $S_{12}$, $T_{12}$, $S_{21}$ and~$T_{21}$ are inessential.
Since~$\sE$ is an operator ideal, it follows that the operator
\[
ST - \begin{bmatrix} 0 &  0\\ 0 & S_{22}T_{22} \end{bmatrix} \] is inessential, and therefore, using \Cref{R:ss_iness}\ref{inessperturb}, we obtain
\begin{align*}
I_\cX-ST\in\Phi_k(\cX)\ &\Longleftrightarrow\  \begin{bmatrix} I_{\cX_1} & 0\\ 0 & I_{\cX_2} - S_{22}T_{22}\end{bmatrix}\in\Phi_k(\cX)\\
&\Longleftrightarrow\ I_{\cX_2}-S_{22}T_{22}\in\Phi_k(\cX_2),
\end{align*}
so $k\in \mathbb{I}_{\normalfont{\text{SC}}}(\cX_2,\cY_2)$ by another application of \Cref{T:SCfredholmChar}.
\end{proof}

We can now  complete the proof of \Cref{P:beyondprojincomp}.

\begin{proof}[Proof of Proposition~{\normalfont{\ref{P:beyondprojincomp}}}]
We have already proved Equation~\eqref{nonPIeaeform} on page~\pageref{proofofEq:nonPIeaeform}.

To prove~\eqref{nonPIscform}, suppose that each of the pairs $(\cX_1,\cX_2)$, $(\cY_1,\cY_2)$ and $(\cX_1,\cY_1)$ is essentially incomparable. Then the hypothesis of
  \Cref{L:essincomp}\ref{L:essincomp2} is satisfied because $\cX_2\cong\cY_2$, so  $\mathbb{I}_{\normalfont{\text{SC}}}(\cX,\cY)= \mathbb{I}_{\normalfont{\text{SC}}}(\cX_2,\cY_2)$, and therefore $\operatorname{sc}(\cX,\cY) = \operatorname{sc}(\cX_2,\cY_2)$. Moreover, \Cref{P:SCcomplSubspace} implies that $\operatorname{sc}(\cX_2,\cY_2) = \eae(\cX_2,\cY_2)$, which completes the proof of~\eqref{nonPIscform}.

Combining~\eqref{nonPIeaeform} and~\eqref{nonPIscform}, we see that  $\eae(\cX,\cY)=\scr(\cX,\cY)$ if and only if $\gcd(\eae(\cX_1,\cY_1),\eae(\cX_2,\cY_2)) = \eae(\cX_2,\cY_2)$, which is equivalent to saying that  $\eae(\cX_2,\cY_2)$ divides $\eae(\cX_1,\cY_1)$.
\end{proof}

\section{The Gowers--Maurey--Aiena--Gonz\'{a}lez--Ferenczi cycle of ideas and the proofs of Theorems~\ref{T:essinNew}, \ref{thm:improjEAE=SC} and~\ref{T:improjUnited}}\label{sec:improj}

\noindent The Banach space that Aiena and Gonz\'{a}lez used in~\cite{AG} to show that projective incomparability does not imply essential incomparability is the so-called ``shift space''  constructed by Gowers and Maurey in \cite[\S{}(4.2)]{gm2}.
Refining the approach of Aiena and Gonz\'{a}lez, Ferenczi \cite[Section 4]{F}   has more recently used this space to prove that there is no largest proper operator ideal, thereby solving a famous open problem going back to Pietsch's monograph \cite{pie}.

The proofs of Theorems~\ref{thm:improjEAE=SC} and~\ref{T:improjUnited} are inspired by this body of work, notably the proof of \Cref{lemma_proj_incNew}\ref{lemma_proj_incNew2}. However, the shift space itself will not suffice for our purposes; we need to work with a larger family of ``higher-order shift spaces''  which Gowers and Maurey also constructed in \cite{gm2}. We shall now give a brief introduction to this family.

Following  the terminology introduced in \cite[page~549]{gm2}, for two infinite subsets $A=\{a_1<a_2<\cdots\}$ and $B=\{ b_1<b_2<\cdots\}$ of~$\N$, we define the associated \emph{spread} $S_{A,B}$ to be the linear map on~$c_{00}$ determined by
\[ S_{A,B}e_j = \begin{cases} e_{b_k}\ &\text{if}\ j=a_k\ \text{for some}\ k\in\N,\\ 0\ &\text{otherwise,} \end{cases} \]
where $(e_n)_{n\in\N}$ denotes the standard unit vector (Hamel) basis for~$c_{00}$.
 Let \mbox{$k_0\in\N_0$}, and  set $\mathscr{M}_{k_0} = \{ [k_0m+1,\infty)\cap\N : m\in\N_0\}$.
  Then
\mbox{$\mathscr{S}_{k_0} = \{ S_{A,B} : A,B\in\mathscr{M}_{k_0}\}$}
is a ``proper set of spreads'' as defined in \cite[page~549]{gm2}, but since we do not need the precise definition of this term in the sequel, we omit the details. The important point is that, by \cite[Theorem~5]{gm2}, $\mathscr{S}_{k_0}$ induces a Banach space, which we shall call the \emph{$k_0$-fold Gowers--Maurey shift space} and denote by~$\GM(k_0)$. (It is denoted~$X(\mathscr{S}_{k_0})$ in~\cite{gm2}.)

As already mentioned, Gowers and Maurey defined and investigated this family of Banach spaces in~\cite{gm2}. More precisely, they studied the space~$\GM(0)$ in~\cite[\S{}(4.1)]{gm2}, $\GM(1)$ in~\cite[\S{}(4.2)]{gm2} and $\GM(2)$ in~\cite[\S{}(4.3)]{gm2}, before outlining  the general case of~$\GM(k_0)$ for $k_0\ge3$  in the final paragraph of \cite[\S{}(4.3)]{gm2}. The following theorem summarizes the results from~\cite{gm2} that we require, together with the necessary notation and terminology.

\begin{thm}[Gowers and Maurey]\label{T:GMspace} Let $k_0\in\N_0$.
\begin{enumerate}[label={\normalfont{(\roman*)}}]
\item\label{T:GMspace1} The Banach space $\GM(k_0)$ has a normalized Schauder basis $(e_n)_{n\in\N}$ which admits an iso\-metric $k_0$-fold right shift operator $R_{k_0}\in\sB(\GM(k_0))$ given by $R_{k_0}e_n= e_{n+k_0}$ for every $n\in\N$, with left inverse $L_{k_0}\in\sB(\GM(k_0))$ given by $L_{k_0}e_n=0$ for $n\le k_0$ and $L_{k_0}e_n = e_{n-k_0}$ for $n > k_0$.
\item\label{T:GMspace2}  The Banach space
  $\GM(k_0)$ satisfies a lower $f$-estimate for the function $f(t) = \log_2(t+1);$ that is,
  \[   \log_2(n+1)\,\biggl\lVert\sum_{k=1}^n x_k\biggr\rVert\ge \sum_{k=1}^n \lVert x_k\rVert \] for every $n\in\N$ and vectors $x_1,\ldots,x_n\in\GM(k_0)$ which are consecutive in the sense that there are integers $0\le m_0<m_1<\cdots <m_n$ such that $x_k\in\operatorname{span}\{ e_j : m_{k-1}<j\le m_k\}$ for each  $1\le k\le n$.
\item\label{T:GMspace3} The Banach space $\GM(k_0)$ is indecomposable; that is, every complemented subspace of~$\GM(k_0)$ is either finite-di\-men\-sional or finite-co\-di\-men\-sional.
\item\label{T:GMspace4} The index~$\gamma$ introduced in Remark~{\normalfont{\ref{R:eae_gcd}}} is given by  $\gamma(\GM(k_0)) = k_0$, and
  $\GM(k_0)$ is not isomorphic to any of its subspaces of infinite codimension.
  Therefore a closed subspace~$\cW$ of $\GM(k_0)$ is isomorphic to~$\GM(k_0)$ if and only if \[ \dim \GM(k_0)/\cW\in k_0\N_0. \]
\item\label{T:GMspace5} The Banach space $\GM(k_0)$ contains no unconditional basic sequences.
 \end{enumerate}
\end{thm}
\begin{proof}
Parts~\ref{T:GMspace1} and~\ref{T:GMspace2} follow from \cite[Theorem~5]{gm2} and the definitions and conventions that it relies on.

  For $k_0=0$, parts \ref{T:GMspace3}--\ref{T:GMspace5} are all derived in \cite[\S{}(4.1)]{gm2}. (Note in this context that $L_0=R_0 = I_{\GM(0)}$.)   Hence it remains to consider $k_0\in\N$.

  \ref{T:GMspace3}. This result is contained in the proof of \cite[Theorem~13]{gm2} for $k_0=1$, with \cite[Remarks, page~559]{gm2} explaining how to  generalize that proof  to arbitrary~$k_0\ge 2$.

  \ref{T:GMspace4}. This result  is a restatement of \cite[Theorem~16]{gm2}  for $k_0=1$. For $k_0\ge 2$, it follows from \cite[Theorem~19]{gm2} and \cite[Remarks, page~559]{gm2}.

\ref{T:GMspace5}. This result is proved in the final paragraph on \cite[page~567]{gm2}.
\end{proof}

\begin{cor}\label{GMtotincompuncondbasis} Let $\cY$ be a Banach space with an unconditional basis. Then, for every $k_0\in\N_0$, $\GM(k_0)$ and~$\cY$ are totally incomparable.
\end{cor}

\begin{proof} As explained in the comment after \cite[Problem~1.d.5]{LT1}, every closed, infinite-dimensional subspace of~$\cY$ contains an  unconditional basic sequence. Hence  \Cref{T:GMspace}\ref{T:GMspace5} implies that no such subspace embeds isomorphically into~$\GM(k_0)$.
\end{proof}

 Using these results, we can  easily prove \Cref{T:essinNew}.

\begin{proof}[Proof of Theorem~{\normalfont{\ref{T:essinNew}}}]
  \ref{T:essinNew1}. We must show that $\mathbb{I}_{\text{SC}}(\cX,\cY)= \{0\}$. This was already proved in \cite[Theorem~2.1(2)]{tHMRR19}, but we would like to point out that it is also an almost immediate consequence of \Cref{T:SCfredholmChar}. Indeed, the  essential incomparability of~$\cX$ and~$\cY$ means that $I_\cX-ST\in\Phi(\cX)$ for every $S\in\sB(\cY,\cX)$ and $T\in\sB(\cX,\cY)$, and \Cref{R:ss_iness}\ref{inessperturb} shows that $i(I_\cX-ST)=0$. Therefore  condition~\ref{T:SCfredholmChar1} in \Cref{T:SCfredholmChar} is satisfied only for $k=0$, so $\opSC_k(\cX,\cY)=\emptyset$ for every $k\in\Z\setminus\{0\}$.

  \ref{T:essinNew2}. Set $\cX = \GM(k_0)$,    and let~$\cY$ be a Banach space which has an unconditional basis  and is isomorphic to its   hyperplanes, so that $\gamma(\cY)=1$.  (For instance, \mbox{$\cY = \ell_2$} has these properties.)  \Cref{GMtotincompuncondbasis} shows that~$\cX$ and~$\cY$ are  totally incomparable and  therefore
  essentially incomparable. Moreover, we have $\gamma(\cX) = k_0$ by \Cref{T:GMspace}\ref{T:GMspace4},  so $\eae(\cX,\cY)=\lcm(k_0,1)=k_0$ by~\eqref{eq:eaeDefnAlt}.
 \end{proof}

While the above proof did not involve any ideas from~\cite{AG} or~\cite{F}, the proofs of Theorems~\ref{thm:improjEAE=SC} and~\ref{T:improjUnited} will, namely in the shape of part~\ref{lemma_proj_incNew2} of the following lemma.

\begin{lem}\label{lemma_proj_incNew} Let $k_0\in\N$.
\begin{enumerate}[label={\normalfont{(\roman*)}}]
\item\label{lemma_proj_incNew1}  Suppose that~$\cX_1$ and~$\cY_1$ are essentially incomparable Banach spaces with unconditional bases and that~$\cY_2$ is a closed, infinite-di\-men\-sional and infinite-co\-di\-men\-sional sub\-space of~$\GM(k_0)$. Then the Banach spaces~$\cX_1\oplus\GM(k_0)$ and~$\cY_1\oplus\cY_2$ are projectively incomparable.
\item\label{lemma_proj_incNew2} The Banach space~$\GM(k_0)$ contains a  closed, in\-finite-di\-men\-sional and in\-finite-co\-di\-men\-sional sub\-space~$\cY_2$ such that $\opSC_{k_0}(\GM(k_0),\cY_2)\ne\emptyset$.
\end{enumerate}
\end{lem}

\begin{rem}\label{R:X1=0}  \Cref{lemma_proj_incNew}\ref{lemma_proj_incNew1} is also true for $\cX_1=\{0\}$ (even though it may be debatable whether this space has an unconditional basis). This observation will be important  in the proofs of Theorems~\ref{thm:improjEAE=SC}\ref{T:improjEAE=SCv1}
and~\ref{T:improjUnited}\ref{T:improjV1}. The conscientious reader can check that the proof which we are about to present remains valid for $\cX_1=\{0\}$.
\end{rem}

\begin{proof}[Proof of Lemma~{\normalfont{\ref{lemma_proj_incNew}\ref{lemma_proj_incNew1}}}] To unify notation, set $\cX_2 = \GM(k_0)$. The proof is by con\-tra\-dic\-tion, so assume that~$\cX_1\oplus\cX_2$ contains an infinite-dimensional, complemented sub\-space~$\cW$ which is isomorphic to a  complemented sub\-space~$\cZ$ of~$\cY_1\oplus\cY_2$. \Cref{GMtotincompuncondbasis} shows that each of the pairs~$(\cX_1,\cX_2)$ and $(\cY_1,\cY_2)$ is totally in\-comparable, so a theorem of Edel\-stein and Woj\-taszczyk (see \cite[Theorem~3.5]{EW}, or \cite[Theorem~2.c.13]{LT1} for an exposition) implies that~$\cW\cong\cW_1\oplus\cW_2$ and~$\cZ\cong\cZ_1\oplus\cZ_2$, where~$\cW_j$  and~$\cZ_j$ are complemented subspaces of~$\cX_j$ and~$\cY_j$, respectively, for $j\in\{1,2\}$.
  Take  an isomorphism
  \[ U = \begin{bmatrix} U_{11} & U_{12}\\ U_{21} & U_{22}\end{bmatrix}\colon \cW_1\oplus\cW_2\to \cZ_1\oplus \cZ_2. \]
 The hypothesis that~$\cX_1$ and~$\cY_1$ are essentially incomparable  implies that the opera\-tor~$U_{11}$ is inessential because essential incomparability clearly passes to complemented subspaces. Moreover,  $U_{12}$ and~$U_{21}$  are strictly singular and therefore in\-essen\-tial by \Cref{GMtotincompuncondbasis} and \Cref{R:incompnotions}\ref{R:incompnotions1}. Consequently
 \[ \begin{bmatrix} 0 & 0\\ 0 & U_{22}\end{bmatrix} =  U - \begin{bmatrix} U_{11} & U_{12}\\ U_{21} & 0\end{bmatrix} \]
   is an inessential perturbation of the isomorphism~$U$ and hence a Fredholm ope\-ra\-tor. This implies that~$U_{22}$ is a Fredholm operator and that~$\cW_1$ is finite-di\-men\-sional, so~$\cW_2$ must be infinite-dimensional. Since it is complemented in~$\cX_2$, \Cref{T:GMspace}\ref{T:GMspace3} shows that~$\cW_2$ has finite codimension in~$\cX_2$.

   Choose a closed subspace~$\cW_3$ of finite codimension in~$\cW_2$ such that \[ \cW_3\cap \ker U_{22}=\{0\}\qquad\text{and}\qquad \dim \cX_2/\cW_3\in k_0\N_0. \]  Then  $\cW_3\cong\cX_2$ by \Cref{T:GMspace}\ref{T:GMspace4}, and the restriction of~$U_{22}$ to~$\cW_3$ is an isomorphic embedding into $\cZ_2\subseteq \cY_2$. However, another application of  \Cref{T:GMspace}\ref{T:GMspace4} shows that  no such embedding exists because~$\cY_2$ has infinite co\-dimen\-sion in~$\cX_2$.
\end{proof}

In order to prove the second part of \Cref{lemma_proj_incNew}, we require two lemmas. The statement of the first of these involves the following standard piece of terminology. An operator $T\in\sB(\cX,\cY)$ (where~$\cX$ and~$\cY$ can be any Banach spaces) is \emph{bounded below} if there exists $\epsilon>0$ such that $\lVert Tx\rVert\ge\epsilon\lVert x\rVert$ for every $x\in\cX$. This is equivalent to saying that~$T$ is injective and has closed range, or in other words that~$T$ is an isomorphic embedding.

\begin{lem}\label{lemma_F}
For every $k_0\in\N$,   the operator $I_{\GM(k_0)}-L_{k_0}\in\sB(\GM(k_0))$ is injective, but not bounded below. Consequently its range is not closed  in~$\GM(k_0)$.
\end{lem}

\begin{proof}
  The proof is a simple variant of an argument given by Ferenczi in
  the text preceding \cite[Proposition~16]{F}. First, to show that
  $I_{\GM(k_0)}-L_{k_0}$ is injective, suppose that
  $x=\sum_{j=1}^\infty a_je_j\in\ker(I_{\GM(k_0)}-L_{k_0})$.  Then we
  have
\[
0 = \sum_{j=1}^\infty a_je_j - \sum_{j=1}^\infty a_{j+k_0}e_j,
\]
so $a_j = a_{j+k_0}$ for each $j\in\N$. By induction, we deduce that $a_j = a_{j+mk_0}$ for each $m\in\N$. Keeping $j$ fixed and letting $m\to\infty$, we have $a_{j+mk_0}\to0$, so $a_j=0$. Since this is true for every $j\in\N$, we conclude that $x=0$.

Second,   to verify that  $I_{\GM(k_0)}-L_{k_0}$ is not bounded below, we consider the vector $w_n=\sum_{j=1}^n e_j\in \GM(k_0)$ for $n>k_0$. \Cref{T:GMspace}\ref{T:GMspace2} implies that
  \[
  \log_2(n+1)\,\lVert w_n\rVert\ge \sum_{j=1}^n\lVert e_j\rVert = n,
  \]
  while
  \[ \lVert (I_{\GM(k_0)}-L_{k_0})w_n\rVert = \biggl\|\sum_{j=n-k_0+1}^ne_j\biggr\|\le k_0, \]
so
\[ \frac{\lVert (I_{\GM(k_0)}-L_{k_0})w_n\rVert}{\lVert w_n\rVert}\le \frac{k_0\log_2(n+1)}{n}\to0\quad\text{as}\quad n\to\infty. \] Consequently $I_{\GM(k_0)}-L_{k_0}$ is not bounded below.
\end{proof}

The other lemma that we require  originates in the work of Lebow and Schechter \cite[Theorem~5.4]{LS}.

\begin{lem}\label{lemma_LS}
Let~$\cX$ and~$\cY$ be Banach spaces, and suppose that $A\in\sB(\cX,\cY)$ is an operator whose range is not closed. Then, for every $\epsilon>0$, there exists a nuclear operator \mbox{$B\in \sB(\cX,\cY)$} such that $\lVert B\rVert<\epsilon$ and the closure of the range of the operator $A-B$ has infinite codimension in~$\cY$.
\end{lem}

\begin{proof}
One can be prove this lemma by following the steps of the proof of \cite[Theorem~5.4]{LS}, starting in line 4 with the hypothesis that the range of the operator~$A$ is not closed. The only modification required is that to ensure that the nuclear operator $B=\sum_{k=1}^\infty(A'y_k')\otimes y_k$ has norm less than~$\epsilon$, we must replace the third inequality in \cite[Equation (5.4)]{LS} with the estimate $\lVert A'y_k'\rVert<\epsilon/2^ka_k$.
\end{proof}

\begin{proof}[Proof of Lemma~{\normalfont{\ref{lemma_proj_incNew}\ref{lemma_proj_incNew2}}}]
Combining Lemmas~\ref{lemma_F} and~\ref{lemma_LS}, we can find a nuclear operator $B\in\sB(\GM(k_0))$ such that the closed subspace
\begin{equation*}
\cY_2 = \overline{\ran(I_{\GM(k_0)} - L_{k_0} - B)}
\end{equation*}
has infinite codimension in~$\GM(k_0)$.

To show that $\opSC_{k_0}(\GM(k_0),\cY_2)\ne\emptyset$,
let $T\in\sB(\GM(k_0),\cY_2)$ denote the operator $I_{\GM(k_0)} - L_{k_0}  - B$ regarded as a map into~$\cY_2$, and let $S\in\sB(\cY_2,\GM(k_0))$ be the natural inclusion map. Then we have $I_{\GM(k_0)} - ST = L_{k_0} + B\in\Phi_{k_0}(\GM(k_0))$ because $L_{k_0}\in\Phi_{k_0}(\GM(k_0))$ and~$B$ is compact. This shows that condition~\ref{T:SCfredholmChar1}  in \Cref{T:SCfredholmChar} is satisfied, and the conclusion follows from  condition~\ref{T:SCfredholmChar3}.

Finally, we observe that~$\cY_2$ must be infinite-dimensional because otherwise~$T$ would be a finite-rank operator, in which case $i(I_{\GM(k_0)} - ST) = 0\ne k_0$.
\end{proof}

\begin{rem} \Cref{lemma_LS}  does not follow from the statement of \cite[Theorem~5.4]{LS} it\-self. In fact, Lebow and Schech\-ter could  have concluded their proof of \cite[Theorem~5.4]{LS}  after its first four lines by invoking  the well-known fact that if $\beta(A)<\infty$ for an operator~$A$ between Banach spaces, then~$A$ has closed range.

  However, as we have seen, the remainder of their proof is very useful for our purposes because it establishes the stronger conclusion stated in \Cref{lemma_LS} that we required to prove \Cref{lemma_proj_incNew}\ref{lemma_proj_incNew2}.  More precisely, what we needed was that we can  perturb the operator~$A$  by an inessential operator~$B$ to obtain that $\overline{\ran(A-B)}$ has infinite codimension in~$\cY$. We did not need that the perturbation~$B$ can be chosen to be nuclear and have arbitrarily small norm; we chose to state those facts simply because they follow automatically from the proof.

We remark that both  Aiena--Gonz\'{a}lez and Ferenczi cite  \cite[Theorem~5.4]{LS} in their work, but as far as we can see, that result does not suffice to give their con\-clusions. Like us, they appear to rely on the stronger statement given in \Cref{lemma_LS}.
\end{rem}

\begin{proof}[Proof of Theorem~{\normalfont{\ref{thm:improjEAE=SC}}}] \ref{T:improjEAE=SCv1}.
 Set  $\cX = \GM(k_0)$ and $\cY=\cY_1\oplus\cY_2$, where~$\cY_1$ is a
 Banach space  which is isomorphic to its hyper\-planes and has an unconditional basis (so for instance we can take $\cY_1=\ell_2$ or $\cY_1=c_0$), and~$\cY_2$ is  the
  closed, in\-finite-di\-men\-sional and in\-finite-co\-di\-men\-sional sub\-space of~$\cX$ constructed in \Cref{lemma_proj_incNew}\ref{lemma_proj_incNew2}. Then~$\cX$ and~$\cY$ are projectively incomparable, as observed in \Cref{R:X1=0}. Moreover,
  \Cref{T:GMspace}\ref{T:GMspace4} shows that $\gamma(\cX) = k_0$, while $\gamma(\cY)=1$ because~$\cY_1$ being isomorphic to its hyper\-planes implies that the same is true for~$\cY$.     Therefore $\eae(\cX,\cY) = \lcm(k_0,1)=k_0$ by~\eqref{eq:eaeDefnAlt}.

  In view of \Cref{P:SC1806}, this means that $\operatorname{sc}(\cX,\cY)$ is a multiple of~$k_0$. Hence, to show that $\operatorname{sc}(\cX,\cY) = k_0$, it will suffice to show that $\opSC_{k_0}(\cX,\cY)\ne\emptyset$, which follows by combining \Cref{L:essincomp}\ref{L:essincomp1} with the fact that  $\opSC_{k_0}(\cX,\cY_2)\ne\emptyset$.

\ref{T:improjEAE=SCv2}. This proof is a slightly more elaborate variant of the proof of~\ref{T:improjEAE=SCv1} that we have just given. We begin by choosing two distinct spaces~$\cX_1$ and~$\cY_1$ from the family $\{\ell_p : 1\le p<\infty\}\cup\{c_0\}$, so that~$\cX_1$ and~$\cY_1$ are isomorphic to their hyperplanes,  have unconditional bases and
are totally incomparable (as observed in \cite[page~75]{LT1}, for instance).
Set $\cX_2 = \GM(k_0)$, and let~$\cY_2$ be the subspace of~$\cX_2$ constructed in \Cref{lemma_proj_incNew}\ref{lemma_proj_incNew2},  as in the first part of the proof.
Then $\cX=\cX_1\oplus\cX_2$ and $\cY=\cY_1\oplus\cY_2$ are projectively in\-comparable by \Cref{lemma_proj_incNew}\ref{lemma_proj_incNew1}, and $\gamma(\cX)=\gamma(\cY)=1$ because~$\cX_1$ and~$\cY_1$ are isomorphic to their hyperplanes, so  $\eae(\cX,\cY) = 1$ by~\eqref{eq:eaeDefnAlt}.

It remains to verify  that $\operatorname{sc}(\cX,\cY) = k_0$.
As remarked above, and by \Cref{GMtotincompuncondbasis},
each of the pairs $(\cX_1,\cY_1)$, $(\cX_1,\cY_2)$ and $(\cX_2,\cY_1)$ is totally incomparable and therefore essentially incomparable, so
\Cref{L:essincomp}\ref{L:essincomp2} shows that
$\mathbb{I}_{\normalfont{\text{SC}}}(\cX,\cY)=\mathbb{I}_{\normalfont{\text{SC}}}(\cX_2,\cY_2)$. On the one hand, we have $k_0\in\mathbb{I}_{\normalfont{\text{SC}}}(\cX_2,\cY_2)$ because $\opSC_{k_0}(\cX_2,\cY_2)\ne\emptyset$, so $k_0\Z\subseteq\mathbb{I}_{\normalfont{\text{SC}}}(\cX_2,\cY_2)$ by \Cref{L:SC}\ref{L:SC2}.
 On the other, \mbox{$\mathbb{I}_{\normalfont{\text{SC}}}(\cX_2,\cY_2)\subseteq\mathbb{I}_\Phi(\cX_2) = k_0\Z$},
where the inclusion is obvious and the equality follows from \Cref{T:GMspace}\ref{T:GMspace4}. Hence  $\mathbb{I}_{\normalfont{\text{SC}}}(\cX,\cY) = \mathbb{I}_{\normalfont{\text{SC}}}(\cX_2,\cY_2) = k_0\Z$, and the conclusion follows.
\end{proof}

The proofs of the two parts of \Cref{T:improjUnited} are somewhat more complicated variants of the proofs of the corresponding parts of \Cref{thm:improjEAE=SC} given above. They involve one additional ingredient, namely Gowers' solution to Banach's hyperplane  problem, which was the first infinite-dimensional Banach space shown not to be isomorphic to its hyperplanes (see \cite{go}, as well as \cite[\S{}(5.1)]{gm2} for further results). The following result summarizes the properties of this space that we require.

\begin{thm}[Gowers] \label{T:GowersHyperplane} There exists an infinite-dimensional, reflexive Banach space~$\Go$ with an unconditional basis such that~$\Go$ fails to be isomorphic to any proper subspace of itself.
\end{thm}

\begin{proof} The only part of this statement  that Gowers did not prove explicitly in~\cite{go} is that~$\Go$ is reflexive. We believe that this fact is known to specialists, but as we have been unable to locate a proof of it in the literature, we outline one here. Since~$\Go$ is not isomorphic to its hyperplanes, it cannot contain any complemented subspace which is isomorphic to its hyperplanes, so in particular no complemented subspace of~$\Go$ is isomorphic to~$c_0$ or~$\ell_1$. Hence, no  subspace of~$\Go$ is isomorphic to~$c_0$ by Sobczyk's Theorem  (see, \emph{e.g.,} \cite[Theorem~2.f.5]{LT1}) or to~$\ell_1$ by a much more recent theorem of Finol and W\'{o}jtowicz~\cite{FW}. (This result was previously stated  without proof in~\cite{Li}.)  Therefore, a classical result of James (see \cite{Ja}, or \cite[Theorem~1.c.12(a)]{LT1} for an exposition) shows  that~$\Go$ is reflexive.
  \end{proof}

\begin{proof}[Proof of Theorem~{\normalfont{\ref{T:improjUnited}}}] \ref{T:improjV1}.
  Following the same approach as in the proof of \Cref{thm:improjEAE=SC}\ref{T:improjEAE=SCv1}, but using different notation, we define $\cX_1 = \GM(k_0)$ and $\cY_1 = c_0\oplus\cY_2$, where~$\cY_2$  is the sub\-space of~$\cX_1$ constructed in  \Cref{lemma_proj_incNew}\ref{lemma_proj_incNew2}. Then, as shown in the proof of \Cref{thm:improjEAE=SC}\ref{T:improjEAE=SCv1}, $\cX_1$ and~$\cY_1$ are projectively incomparable, so~\ref{T:improjV1i} holds, and
\begin{equation}\label{T:improjUnited:eq1}
    \eae(\cX_1,\cY_1) = k_0.
\end{equation}

  Let $\cZ=\Go$ be the Banach space from \Cref{T:GowersHyperplane}. Then $\gamma(\cZ)=0$, so \mbox{$\eae(\cZ,\cZ)=0$}, which verifies~\ref{T:improjV1iv}. Moreover,
  \Cref{GMtotincompuncondbasis} shows that~$\cZ$ is totally incomparable with~$\cX_1$, and therefore also with~$\cY_2$. Since every closed subspace of~$c_0$ contains an isomorphic copy of~$c_0$, while~$\cZ$ is reflexive, $\cZ$ and~$c_0$ are also totally incomparable, and therefore~$\cZ$ and~$\cY_1$ are essentially incomparable. This shows that~\ref{T:improjV1ii} is satisfied.

  It remains to verify~\ref{T:improjV1iii}. By~\ref{T:improjV1ii}, we can apply~\eqref{nonPIeaeform} to calculate $ \eae(\cX,\cY)$ for  $\cX = \cX_1\oplus\cZ$ and $\cY = \cY_1\oplus\cZ$. Using~\eqref{T:improjUnited:eq1}, we obtain
  \begin{equation}\label{T:improjUnited:eq2}  \eae(\cX,\cY) = \gcd(\eae(\cX_1,\cY_1),\eae(\cZ,\cZ)) = \gcd(k_0,0) = k_0.
  \end{equation}

 Finally, we combine \Cref{L:essincomp}\ref{L:essincomp1} with the fact that $\opSC_{k_0}(\cX_1,\cY_2)\ne\emptyset$ to deduce  that $\opSC_{k_0}(\cX,\cY)\ne\emptyset$. In view of~\eqref{T:improjUnited:eq2} and  \Cref{P:SC1806}, this implies that $\operatorname{sc}(\cX,\cY) = k_0$, as we already saw in the proof of \Cref{thm:improjEAE=SC}\ref{T:improjEAE=SCv1}.


 \ref{T:improjV2}.  As above, let $\cZ=\Go$ be the Banach space from \Cref{T:GowersHyperplane} and set $\cY_1 = c_0\oplus \cY_2$, where~$\cY_2$ is the subspace of~$\GM(k_0)$ from  \Cref{lemma_proj_incNew}\ref{lemma_proj_incNew2}, but now define $\cX_1 = \ell_1\oplus\GM(k_0)$. Then~$\cX_1$ and~$\cY_1$ are projectively incomparable by  \Cref{lemma_proj_incNew}\ref{lemma_proj_incNew1}. We showed in the first part of the proof that~$\cY_1$ and~$\cZ$ are essentially incomparable; a similar argument gives the same conclusion for~$\cX_1$ and~$\cZ$. Hence conditions \ref{T:improjV1i}--\ref{T:improjV1iv} are satisfied.

  Arguing as before, we see that  the Banach spaces
 \[ \cX = \cX_1\oplus\cZ = \ell_1\oplus\GM(k_0)\oplus\cZ\qquad\text{and}\qquad \cY = \cY_1\oplus\cZ = c_0\oplus \cY_2\oplus\cZ \]
 satisfy $\gamma(\cX)=\gamma(\cY) = 1$, so that $\eae(\cX,\cY)=1$,  and we also have
  $\opSC_{k_0}(\cX,\cY)\ne\emptyset$, which implies that $k_0\Z\subseteq\mathbb{I}_{\text{SC}}(\cX,\cY)$.
To verify the opposite inclusion, we observe that each of the pairs $(\ell_1, c_0)$, $(\ell_1,\cY_2\oplus\cZ)$ and $(c_0,\GM(k_0)\oplus\cZ)$ is essentially incomparable, so \Cref{L:essincomp}\ref{L:essincomp2} shows that
  \[ \mathbb{I}_{\normalfont{\text{SC}}}(\cX,\cY)=
  \mathbb{I}_{\normalfont{\text{SC}}}(\GM(k_0)\oplus\cZ,\cY_2\oplus\cZ)\subseteq \mathbb{I}_\Phi(\GM(k_0)\oplus\cZ) = k_0\Z, \]
where the final equality follows from \Cref{GMtotincompuncondbasis} and  \Cref{L:essincFredholmIndex}\ref{L:essincFredholmIndex2}. Hence we have $\mathbb{I}_{\normalfont{\text{SC}}}(\cX,\cY)=k_0\Z$, and therefore $\operatorname{sc}(\cX,\cY)=k_0$.
  \end{proof}

\subsection*{Acknowledgements}
We are grateful to Tomasz Kania for having brought the reference~\cite{FW} to our attention.

This work is based on research supported in part by the National Research Foundation of South Africa (NRF) and the DSI-NRF Centre of Excellence in Mathematical and Statistical Sciences (CoE-MaSS). Any opinion, finding and conclusion or recommendation expressed in this material is that of the authors and the NRF and CoE-MaSS do not accept any liability in this regard.

\end{document}